%% file: ex_article.tex
\documentclass[final,hidelinks,onefignum,onetabnum]{siamart251216}

\input{ex_shared}

\ifpdf
\hypersetup{
  pdftitle={Lifting-Free Quadratic Sum-Of-Squares Programming},
  pdfauthor={Gabriel F. Machado, Ross Drummond, and Morgan Jones}
}
\fi

\begin{document}

\maketitle

\begin{abstract}
Quadratic Sum-Of-Squares (QSOS) optimization problems appear in system identification and machine learning, but standard Schur-complement and second-order cone liftings enlarge conic dimensions and create computational bottlenecks for interior-point methods. This paper introduces a lifting-free regularization that preserves the original conic structure by adding a norm penalty to SOS variables, yielding closed-form primal updates and an unconstrained, concave dual with Lipschitz-continuous gradient. Accelerated first-order methods efficiently maximize this dual, and convergence analysis shows non-asymptotic recovery of the solution. Numerical experiments on constrained regression problems show the proposed method can be  40\% faster than existing solvers such as SCS and handle larger problems than MOSEK, with memory scaling only in the number of equality constraints.
\end{abstract}




\section{Introduction}
\label{sec:intro}

Sum-Of-Squares (SOS) optimization \cite{parrilo2000structured,laurent2008sums,blekherman2012} is at the heart of modern polynomial optimization, with applications including control theory \cite{luppi2023data,bramburger2024synthesizing}, system identification \cite{hopkins2017power,meng2020application} and machine learning \cite{d2023higher,bach2024theory,azuma2025rank}. A standard SOS program minimizes a \emph{linear} function subject to constraints on polynomials being sum-of-squares. However, applications such as constrained regression and polynomial system identification \cite{ahmadi2023,machado2024sparse}  naturally involve the minimization of \emph{quadratic} objective functions. This gives rise to the following class of Quadratic Sum-Of-Squares (QSOS) programs,
\begin{equation}
\begin{aligned}
    \min_{\bm{\xi}} \quad &
    \tfrac{1}{2}\bm{\xi}_{\textnormal{f}}^{\top} \bm{Q} \bm{\xi}_{\textnormal{f}}
    + \bm{w}^{\top} \bm{\xi} \\
    \text{s.t.} \quad &
    q_{0,j}(\bm{x}) + \sum_{i=1}^{r} p_{i}(\bm{x}) q_{i,j}(\bm{x}) = 0,
    \quad j=1,\dots,m_{\textnormal{eq}}, \\
    &
    t_{0,j}(\bm{x}) + \sum_{i=1}^{r} p_{i}(\bm{x}) t_{i,j}(\bm{x})
    \ \text{is sum-of-squares}, \quad j=1,\dots,m_{\textnormal{sos}}, \\
    &
    p_{i}(\bm{x})\ \text{is sum-of-squares}, \quad i=\hat{r}+1,\dots,r,
\end{aligned}
\tag{QSOS}\label{opt:QSOS}
\end{equation}
where $\bm{x} \in \mathbb{R}^{n_x}$, $\bm{\xi} =
\begin{bmatrix}
    \bm{\xi}_{\textnormal{f}}^{\top} &
    \bm{\xi}_{\textnormal{sos}}^{\top}
\end{bmatrix}^{\top}
\in \mathbb{R}^{n}$ is the vector formed from the coefficients
$\bm{\xi}_{\textnormal{f}}~\in~\mathbb{R}^{n_{\textnormal{f}}}$ of polynomials
$p_{i}(\bm{x})$ for $i=1,\dots,\hat{r}$ and of coefficients
$\bm{\xi}_{\textnormal{sos}} \in \mathbb{R}^{n-n_{\textnormal{f}}}$ of
sum-of-squares polynomials $p_{i}(\bm{x})$ for $i=\hat{r}+1,\dots,r$,
$q_{i,j}(\bm{x})$ and $t_{i,j}(\bm{x})$ are given polynomials with constant
scalar coefficients. In this formulation, $\bm{w}$ is a $n$-dimensional vector,
$\bm{Q}$ is a symmetric $n_{\textnormal{f}} \times n_{\textnormal{f}}$ matrix,
and both represent the weighting factors of the quadratic objective function.

In SOS optimization, a linear SOS program is recovered when $\bm{Q} = \bm{0}$, and this can be parsed into a Semidefinite Programming (SDP) problem
\cite{lofberg2009pre,blekherman2012,jagt2022efficient}, a class of
linear Conic Programming (CP) problems with many important applications
\cite{wolkowicz2012handbook,majumdar2020recent}. Many methods have now been
proposed to solve SDP problems; including (i) path-following (barrier function)
methods \cite{nesterov1997self,nesterov1998primal} and (ii) homogeneous self-dual embedding (HSDE) \cite{ye1994nl}, commonly implemented using interior-point methods (IPMs)
\cite{vandenberghe1996semidefinite,helmberg1996interior,wright1997primal,alizadeh1998primal,nemirovski2008interior}. Whilst widely adopted, the computational time
and memory complexity of IPMs scale poorly with the number of decision variables
and constraints, as forming and factorizing large (and possibly dense) matrix
variables can be numerically expensive \cite{majumdar2020recent}.

First-order solvers based on augmented Lagrangian methods have been proposed for large SDP problems (see Section~12.2 in \cite{fletcher2013practical}). Methods such as the Alternating Direction Method of Multipliers (ADMM) \cite{glowinski1975approximation,gabay1976dual,gabay1983chapter} decompose the SDP into simpler subproblems, and are widely used when solving SDP problems with first-order methods \cite{wen2010alternating,o2016conic}. Unlike IPMs, first-order methods may not require significant computational effort and storage per iteration, allowing them to solve large problems faster, though potentially yielding inaccurate solutions \cite{parikh2014proximal}.

Quadratic Sum-Of-Squares (QSOS) programs can be written as Quadratic Conic Programming (QCP) problems with convex quadratic objectives and feasible sets defined by affine constraints and positive semidefinite cones. One way of solving such QCP problems involves a reduction to a conic optimization problem with a linear objective and a lifted constraint space. This is done by eliminating the quadratic objective and including a quadratic inequality constraint, transformed into either a second-order cone constraint or a semidefinite constraint. Both lifting approaches increase the decision variable space and introduce additional conic constraints. In high-dimensional polynomial optimization problems, where polynomial variables have large degree in many variables, the number of decision variables grows combinatorially. Consequently, second-order cone constraints increase the size of linear systems solved at each iteration, and large positive semidefinite constraints require computationally expensive factorizations. These limitations motivate an alternative method that solves QCP problems directly, without increasing the cone dimension or incurring the cost of involved matrix operations.

Several specialized solvers, including COSMO \cite{garstka2021cosmo}, SCS \cite{o2021operator}, and Clarabel \cite{goulart2024clarabel}, can directly handle quadratic objectives without reducing to a linear CP. COSMO extends operator splitting methods \cite{stellato2020osqp} and applies ADMM \cite{neal2011distributed}. Clarabel builds upon solving a specialized HSDE \cite{andersen1999homogeneous} with an IPM, and SCS also uses HSDE, but with ADMM. For large problems with positive semidefinite constraints, their speed-ups rely not only on numerical methods but also on exploiting structure or sparsity patterns \cite{sun2014decomposition,zheng2017exploiting,zheng2020chordal}. These solvers inherit structural limitations that become pronounced in large SOS problems, such as lifted spaces due to the addition of slack variables (see \cref{opt:standard-QCP}). COSMO relies on operator splitting and requires carefully engineered decompositions, while Clarabel and SCS depend on solving a HSDE whose size grows with the conic structure. All three methods repeatedly apply PSD projections or factorization-based updates that scale poorly with the number and dimension of SOS constraints. Their performance is highly sensitive to problem conditioning and sparsity patterns, and they may fail or become prohibitively slow on dense or high-dimensional SOS instances. These limitations motivate the need for a lifting-free approach that avoids augmented Lagrangians, slack variables, and large conic embeddings, and instead provides closed-form primal solutions together with non-asymptotic convergence guarantees.

\textbf{Main Contributions.} This paper presents a novel method for solving quadratic SOS programs that owns:
\begin{enumerate}
    \item Lifting-free regularization preserving conic structure and yielding closed-form primal solutions (see \Cref{sec:theory});
    \item Explicit non-asymptotic convergence guarantees (\cref{thm:suboptimality-fixed-rho}) using Nesterov’s accelerated gradient ascent (\cref{alg:first-order-solver});
    \item Practical scalability and reliability demonstrated on constrained regression problems, outperforming SCS and MOSEK in failure rate, speed, and memory usage (see \Cref{sec:experiments}).
\end{enumerate}

\textbf{Related Work.} QCP problems appear in the nearest correlation matrix problem \cite{higham2002computing} and in Semidefinite Least-Squares (SDLS) problems \cite{malick2004dual}. Algorithm 3.3 in \cite{higham2002computing} can be interpreted as a simple dual gradient method, whereas Algorithm 1 in \cite{malick2004dual} uses a quasi-Newton optimizer. In \cite{henrion2011projection}, Algorithm 2.1 is a dual-based first-order method introducing the SDLS MATLAB package \cite{henrion2007sdls}, consisting of projection steps for conic feasibility problems such as SOS decompositions. These works compare Schur-complement and second-order cone liftings and report that both approaches can be inferior due to dimensional growth. Higham’s analysis is performed on a problem already in regularized form, whereas our method applies regularization directly to general programs in the form of \cref{opt:QSOS}. Our convexity and smoothness arguments build upon the dual-based analysis framework of \cite{malick2004dual}, adapted to \cref{opt:QSOS}. Feasibility-oriented SOS methods \cite{henrion2011projection} differ substantially: feasibility problems do not require handling regularization in the objective, and their quasi-Newton scheme contrasts with our use of accelerated gradient methods enabling non-asymptotic convergence guarantees.

Regularization methods for large-scale SDP problems stabilize the linear SDP problem with small perturbations in the data, typically adding quadratic regularization terms to the objective or constraints. Connections between primal Moreau-Yosida regularization and dual augmented Lagrangian methods appear in \cite{malick2009regularization}, and augmented primal–dual methods using regularization appear in \cite{florian2008augmented}. These approaches motivate our lifting-free regularization framework, which dualizes only affine constraints, avoids augmented Lagrangians, and yields closed-form primal updates together with non-asymptotic convergence guarantees.

\textbf{Paper Structure.} \Cref{sec:lifting} reviews standard reformulations of quadratic SOS programs and their scalability limitations. \Cref{sec:theory} introduces the regularized problem \cref{opt:Reg-QCP}, derives its dual, and presents \cref{alg:first-order-solver}. \Cref{sec:convergence} provides convergence analysis. \Cref{sec:experiments} presents numerical experiments. \Cref{sec:conclusion} concludes and outlines future work.

\textbf{Notation.} The $n$-dimensional Euclidean space is denoted by $\mathbb{R}^{n}$, and the set of $n\times n$ symmetric matrices by $\mathbb{S}^{n}$. The cone of positive semidefinite matrices is $\mathbb{S}_{+}^{n} \subset \mathbb{S}^{n}$. For a vector $\bm{v}\in\mathbb{R}^{n}$, $\|\bm{v}\|_{2}$ is the Euclidean norm and $\|\bm{v}\|_{\infty}$ the infinity norm. For a matrix $\bm{X}\in\mathbb{R}^{m\times n}$, $\|\bm{X}\|_{2}$ denotes the spectral norm, $\|\bm{X}\|_{\textnormal{F}}$ the Frobenius norm, and $\sigma_{\textnormal{min}}(\bm{X})$ denotes its smallest singular value. The operator $\operatorname{vec}(\bm{X})$ stacks the columns of $\bm{X}$ into a vector. For symmetric $\bm{X} \in \mathbb{S}^{n}$, $\operatorname{svec}(\bm{X})$ extracts the upper triangular part scaled by $\sqrt{2}$ for off-diagonal entries, and $\operatorname{smat}(\cdot)$ is its inverse. The identity $\|\bm{X}\|_{\textnormal{F}} = \|\operatorname{svec}(\bm{X})\|_{2}$ shows that $(\mathbb{S}^{n}, \Vert \cdot \Vert_{\textnormal{F}})$ is isometric to $(\mathbb{R}^{n(n+1)/2}, \Vert \cdot \Vert_{2})$. The Euclidean projection onto a closed convex set $C \subset \mathbb{R}^{d}$ is $\mathcal{P}_{C}(\bm{z}) = \arg\min_{\bm{x}\in C} \|\bm{x}-\bm{z}\|_{2}$. When $C = \mathbb{S}_{+}^{n}$, the projection is taken in $(\mathbb{S}^{n}, \Vert \cdot \Vert_{\textnormal{F}})$ and coincides with the Euclidean projection after applying $\operatorname{svec}(\cdot)$. For a convex cone $K$, its polar cone is $K^{\circ} = \{\bm{y} : \bm{y}^{\top}\bm{x} \ge 0,\ \forall\,\bm{x}\in K\}$.

\section{Reformulating Quadratic SOS Programs} \label{sec:lifting}
When $\bm{Q} = \bm{0}$ in \cref{opt:QSOS}, linear SOS programs can be parsed into Semidefinite Programming (SDP) problems \cite{parrilo2000structured,parrilo2003semidefinite,blekherman2012}, and the equivalence between SOS and positive semidefinite constraints underlies the standard conic formulation \cite{lofberg2009pre,jagt2022efficient},
\[
    \bm{A}\bm{\xi} = \bm{b}, \qquad \bm{\xi} \in K,
\]
where $\bm{A} \in \mathbb{R}^{m \times n}$, $\bm{b} \in \mathbb{R}^{m}$, $K \subset \mathbb{R}^{n}$ is a nonempty closed convex cone, and $\bm{\xi} =
\begin{bmatrix}
    \bm{\xi}_{\textnormal{f}}^{\top} &
    \bm{\xi}_{\textnormal{sos}}^{\top}
\end{bmatrix}^{\top}
\in \mathbb{R}^{n}$ collects free polynomial coefficients and SOS polynomials coefficients.

Applying this representation to \cref{opt:QSOS} yields a Quadratic Conic Program (QCP) of the form
\begin{equation}
\begin{aligned}
    \min_{\bm{\xi}} \quad &
    \tfrac{1}{2}\bm{\xi}_{\textnormal{f}}^{\top}\bm{Q}\bm{\xi}_{\textnormal{f}}
    + \bm{w}^{\top}\bm{\xi} \quad 
    \text{s.t.} \quad \bm{A}\bm{\xi} = \bm{b}, \quad \bm{\xi} \in K.
\end{aligned}
\tag{QCP}\label{opt:QCP}
\end{equation}

\textbf{Lifting Quadratic SOS Programs into Linear Conic Programs.}
A common approach is to lift the quadratic objective into a linear conic form by introducing an auxiliary variable $\gamma \ge 0$ satisfying $\tfrac{1}{2}\bm{\xi}_{\textnormal{f}}^{\top}\bm{Q}\bm{\xi}_{\textnormal{f}} \le \gamma$. This inequality can be encoded either as a second-order cone constraint or as a semidefinite constraint.

For the SOCP lifting, the Cholesky factorization $\bm{Q} = \bm{E}\bm{E}^{\top}$ yields
\begin{equation}
    \bm{\xi}_{\textnormal{f}}^{\top}\bm{Q}\bm{\xi}_{\textnormal{f}}
    = \bm{\xi}_{\textnormal{f}}^{\top}\bm{E}\bm{E}^{\top}\bm{\xi}_{\textnormal{f}}
    = \| \bm{E}^{\top}\bm{\xi}_{\textnormal{f}} \|_{2}^{2}
    \le 2\gamma,
    \label{eq:epigraph}
\end{equation}
which defines a rotated second-order cone \cite{lobo1998applications}. The lifted problem becomes
\begin{equation}
\begin{aligned}
    \min_{\gamma,\,\bm{\xi}} \quad & \gamma + \bm{w}^{\top}\bm{\xi} \quad
    \text{s.t.} \quad \bm{A}\bm{\xi} = \bm{b}, \quad \left\|
        \begin{bmatrix}
            2\bm{E}^{\top}\bm{\xi}_{\textnormal{f}} \\
            \gamma - 2
        \end{bmatrix}
      \right\|_{2}
      \le \gamma + 2, \quad \bm{\xi} \in K,\quad \gamma \ge 0.
\end{aligned}
\tag{SOCP}\label{opt:socp}
\end{equation}

Alternatively, the Schur complement encodes the same inequality as the linear matrix inequality \cite{vandenberghe1996semidefinite,calafiore2014optimization}
\[
    2\gamma - \bm{\xi}_{\textnormal{f}}^{\top}\bm{E}\bm{E}^{\top}\bm{\xi}_{\textnormal{f}} \ge 0
    \iff
    \begin{bmatrix}
        2\gamma & \bm{\xi}_{\textnormal{f}}^{\top}\bm{E} \\
        \bm{E}^{\top}\bm{\xi}_{\textnormal{f}} & \bm{I}
    \end{bmatrix}
    \in \mathbb{S}_{+}^{n_{\textnormal{f}}+1},
\]
yielding an SDP lifting with an additional $(n_{\textnormal{f}}+1) \times (n_{\textnormal{f}}+1)$ positive semidefinite constraint,
\begin{equation}
\begin{aligned}
    \min_{\gamma,\,\bm{\xi}} \quad & \gamma + \bm{w}^{\top}\bm{\xi} \quad \text{s.t.} \quad \bm{A}\bm{\xi} = \bm{b}, \quad
    \begin{bmatrix}
        2\gamma & \bm{\xi}_{\textnormal{f}}^{\top}\bm{E} \\
        \bm{E}^{\top}\bm{\xi}_{\textnormal{f}} & \bm{I}
    \end{bmatrix}
    \in \mathbb{S}_{+}^{n_{\textnormal{f}}+1}, \quad \bm{\xi} \in K,\quad \gamma \ge 0.
\end{aligned}
\tag{SDP}\label{opt:schur}
\end{equation}

Both liftings increase the decision variable dimension from 
$n$ to $n + 1$ and introduce conic constraints whose sizes scale as $\mathcal{O}(n)$ for \cref{opt:socp} and $\mathcal{O}(n^{2})$ for \cref{opt:schur}. In high-dimensional polynomial optimization, where the number of coefficients grows combinatorially, these liftings create computational bottlenecks: SOCP constraints enlarge the linear systems solved at each iteration, and SDP constraints require expensive factorizations. Consequently, both approaches become inefficient for large QSOS problems, motivating a lifting-free method that solves \cref{opt:QCP} directly without increasing cone dimension or incurring additional matrix operations.

\textbf{Quadratic Conic Program Solvers.}
Several specialized solvers, including COSMO \cite{garstka2021cosmo}, SCS \cite{o2021operator}, and Clarabel \cite{goulart2024clarabel}, are able to directly handle quadratic objectives without reducing to a linear CP. These methods solve problems of the form
\begin{equation}
\begin{aligned}
    \min_{\bm{\xi},\,\bm{s}} \quad &
    \tfrac{1}{2}\bm{\xi}^{\top}\widetilde{\bm{Q}}\bm{\xi}
    + \bm{w}^{\top}\bm{\xi} \quad
    \text{s.t.} \quad \widetilde{\bm{A}}\bm{\xi} + \bm{s} = \widetilde{\bm{b}}, \quad \bm{s} \in \widetilde{K},
\end{aligned}
\label{opt:standard-QCP}
\end{equation}
typically using operator splitting or homogeneous self-dual embeddings.

Comparing \cref{opt:standard-QCP} with \cref{opt:QCP}, there are key structural differences lying in the quadratic term of the objective and how the constraints are defined. In the first, the matrix $\widetilde{\bm{Q}}$ acts on the full decision variable $\bm{\xi} \in \mathbb{R}^{n}$, whereas in \cref{opt:QCP} the quadratic term $\tfrac{1}{2}\bm{\xi}_{\textnormal{f}}^{\top}\bm{Q}\bm{\xi}_{\textnormal{f}}$ acts only on the free variables $\bm{\xi}_{\textnormal{f}} \in \mathbb{R}^{n_{\textnormal{f}}}$, with the SOS variables $\bm{\xi}_{\textnormal{sos}}^{(j)}$, $j=1,\dots,m_{\textnormal{sos}}$, constrained to a product of positive semidefinite cones. Secondly, the conic constraint $\bm{\xi} \in K$ in \cref{opt:QCP} requires the SOS variables to lie in a product of positive semidefinite cones, which in the standard form must be enforced through the slack variable as $\bm{s} \in \widetilde{K}$. This means the equality constraints must simultaneously encode both the original affine constraints and the positive semidefinite cone membership. The latter requires introducing additional equality constraints, which link the SOS variables to slack variables to the original conic constrained variables.

Despite avoiding explicit SOCP/SDP liftings, these solvers still inherit structural limitations in large SOS problems. Enforcing PSD constraints requires slack variables and repeated projections or factorizations, and the HSDE size grows with the conic structure. As the number and dimension of SOS blocks increase, these operations scale poorly and performance becomes sensitive to conditioning and sparsity patterns. Consequently, COSMO, SCS, and Clarabel may fail or slow down significantly on dense or high-dimensional SOS instances. These limitations motivate our lifting-free approach that preserves the original conic structure, avoids slack variables and augmented Lagrangians, and yields closed-form primal updates together with non-asymptotic convergence guarantees.

\section{Regularization Method for Quadratic SOS Programs}
\label{sec:theory}

The previous section reviewed lifting-based reformulations of quadratic conic programs and highlighted how second-order cone and semidefinite approaches increase conic dimensions and computational cost. This motivates a lifting-free strategy. Instead of adding new constraints, the objective of \cref{opt:QCP} is modified, leading to the regularized problem
\begin{equation}
\begin{aligned}
    \min_{\bm{\xi}} \quad &
    \tfrac{1}{2}\bm{\xi}_{\textnormal{f}}^{\top}\bm{Q}\bm{\xi}_{\textnormal{f}}
    + \bm{w}^{\top}\bm{\xi}
    + \tfrac{\rho}{2}\sum_{j=1}^{m_{\textnormal{sos}}}
      \|\bm{\xi}_{\textnormal{sos}}^{(j)}\|_{2}^{2} \quad \text{s.t.} \quad \bm{A}\bm{\xi} = \bm{b}, \quad \bm{\xi} \in K,
\end{aligned}
\tag{Reg-QCP}\label{opt:Reg-QCP}
\end{equation}
where $\rho>0$ and $\bm{Q}$ is positive definite.

The penalty term regularizes variables associated with positive semidefinite cones and yields a dual problem solvable via first-order methods. As shown in \Cref{sec:convergence}, solutions of \cref{opt:Reg-QCP} converge to those of \cref{opt:QCP}.

The method proceeds by:
\begin{enumerate}
    \item Dualizing only the affine constraints;
    \item Deriving closed-form minimizers of the Lagrangian subproblem;
    \item Establishing concavity, differentiability, and Lipschitz smoothness of the dual function;
    \item Applying accelerated gradient ascent to maximize the dual.
\end{enumerate}

\cref{alg:first-order-solver} provides an overview of our proposed solution method. Our algorithm solves the dual problem in \cref{opt:dual-problem-def} and returns the solutions to \cref{opt:Reg-QCP} analytically. At each iteration $k$, the dual iterate $\bm{\lambda}$ is computed using the gradient in \cref{eq:gradient-dual} evaluated at the extrapolation point $\bm{y}$. The algorithm terminates either when the maximum number of iterations is reached or when the termination criteria are satisfied (see \Cref{sec:convergence}). The next subsections present a sequential derivation of each component of the algorithm.

\begin{algorithm}
\caption{Dual-based accelerated gradient for \cref{opt:Reg-QCP}}
\label{alg:first-order-solver}
\begin{algorithmic}[1]
\State \textbf{Input:} $Q$, $w$, $A$, $b$, $\rho>0$, $\epsilon_{\mathrm{abs}}>0$, $\epsilon_{\mathrm{rel}}>0$, $N\ge 1$, $\eta = 1/L$
\State \textbf{Initialize:} $\lambda^{(0)}\gets 0$, $\bm{y}^{(0)}\gets \lambda^{(0)}$, $t^{(0)}\gets 1$
\For{$k = 0$ to $N-1$}
    \If{$k > 0$}
        \State $t^{(k)} \gets \tfrac12 + \sqrt{\tfrac14 + (t^{(k-1)})^{2}}$
        \State $\lambda^{(k)} \gets \bm{y}^{(k-1)} + \eta\,\nabla g(\bm{y}^{(k-1)})$
        \If{$\nabla g(\bm{y}^{(k-1)})^{\top}(\lambda^{(k)}-\lambda^{(k-1)}) < 0$}
            \State $t^{(k)} \gets 1$
            \State $\lambda^{(k-1)} \gets \lambda^{(k)}$
        \EndIf
        \State $\bm{y}^{(k)} \gets \lambda^{(k)} + \frac{t^{(k-1)}-1}{t^{(k)}}(\lambda^{(k)}-\lambda^{(k-1)})$
    \EndIf
    \State Compute $\nabla g(\bm{y}^{(k)})$ using \cref{eq:gradient-dual}
    \State $r_{\mathrm{prim}}^{(k)} \gets \|\nabla g(\bm{y}^{(k)})\|_{\infty}$
    \State Recover $\bm{\xi}_{\textnormal{f}}^{(k)} = \mathcal{S}_{\bm{\xi}_{\textnormal{f}}}(\bm{y}^{(k)})$ and $\bm{\xi}_{\textnormal{sos}}^{(j,k)} = \mathcal{S}_{\bm{\xi}_{\textnormal{sos}}^{(j)}}(\bm{y}^{(k)})$
    \State Compute $f_{\rho}^{(k)}$ and $g^{(k)}$
    \State $r_{\mathrm{gap}}^{(k)} \gets |\bm{y}^{(k)\top}\nabla g(\bm{y}^{(k)})|$
    \If{$r_{\mathrm{prim}}^{(k)} < \epsilon_{\mathrm{abs}} + \epsilon_{\mathrm{rel}}\max\{\|\bm{y}^{(k)}\|_{\infty},1\}$ \textbf{and}
        $r_{\mathrm{gap}}^{(k)} < \epsilon_{\mathrm{abs}} + \epsilon_{\mathrm{rel}}\max\{|f_{\rho}^{(k)}|,|g^{(k)}|\}$}
        \State \textbf{break}
    \EndIf
\EndFor
\State \textbf{Output:} primal solution from $\bm{y}^{(k)}$ and dual solution $\lambda = \bm{y}^{(k)}$
\end{algorithmic}
\end{algorithm}

\subsection{Partial Lagrangian Dualization}

We define the partial Lagrangian of \cref{opt:Reg-QCP} as
\begin{equation}
    \begin{split}
        \mathcal{L}\left(\bm{\xi},\bm{\lambda}\right) \coloneqq \frac{1}{2}\bm{\xi}_{\textnormal{f}}^{\top}\bm{Q}\bm{\xi}_{\textnormal{f}} + \bm{w}^{\top} \bm{\xi} + \frac{\rho}{2} \sum_{j=1}^{m_{\textnormal{sos}}} \Vert \bm{\xi}_{\textnormal{sos}}^{(j)} \Vert_{\textnormal{2}}^{2} - \bm{\lambda}^{\top}(\bm{A}\bm{\xi} - \bm{b}),
    \end{split}
    \nonumber
\end{equation}
where $\bm{\xi} = \begin{bmatrix}
    \bm{\xi}_{\textnormal{f}}^{\top}  &  \bm{\xi}_{\textnormal{sos}}^{\top}
\end{bmatrix}^{\top}$ is the vector of decision variables (primal variables) and $\bm{\lambda} \in \mathbb{R}^{m}$ is the vector of Lagrange multipliers (dual variables). This is a partial Lagrangian because only the affine constraints are dualized and the conic constraints on the primal variables remain \cite{malick2004dual}.

Conformal decompositions of $\bm{A}$ and $\bm{w}$
\begin{equation}
    \bm{A} = \begin{bmatrix}
    \bm{A}_{\textnormal{f}} & \bm{A}_{\textnormal{sos}}^{(1)} & \cdots & \bm{A}_{\textnormal{sos}}^{(m_{\textnormal{sos}})}
\end{bmatrix}, \quad \bm{w} = \begin{bmatrix}
    \bm{w}_{\textnormal{f}}^{\top} & {\bm{w}_{\textnormal{sos}}^{(1)}}^{\top} & \cdots & {\bm{w}_{\textnormal{sos}}^{(m_{\textnormal{sos}})}}
\end{bmatrix}^{\top}
\end{equation}
separate the partial Lagrangian additively as $\mathcal{L}\left(\bm{\xi},\bm{\lambda}\right) = \mathcal{L}_{\textnormal{f}}(\bm{\xi}_{\textnormal{f}},\bm{\lambda}) + \sum_{j=1}^{m_{\textnormal{sos}}} \mathcal{L}_{\textnormal{sos}}(\bm{\xi}_{\textnormal{sos}}^{(j)},\bm{\lambda})$, with $\mathcal{L}_{\textnormal{f}}$ and $\mathcal{L}_{\textnormal{sos}}$ defined, respectively, as
\begin{subequations}
\begin{align}
        \mathcal{L}_{\textnormal{f}}(\bm{\xi}_{\textnormal{f}},\bm{\lambda})  & \coloneqq \frac{1}{2}\bm{\xi}_{\textnormal{f}}^{\top}\bm{Q}\bm{\xi}_{\textnormal{f}} + \bm{w}_{\textnormal{f}}^{\top} \bm{\xi}_{\textnormal{f}} - \bm{\lambda}^{\top}(\bm{A}_{\textnormal{f}}\bm{\xi}_{\textnormal{f}} - \bm{b}),
    \label{eq:free-Lagrangian} \\
        \mathcal{L}_{\textnormal{sos}}(\bm{\xi}_{\textnormal{sos}}^{(j)},\bm{\lambda})  & \coloneqq \frac{\rho}{2} \Vert \bm{\xi}_{\textnormal{sos}}^{(j)} \Vert_{\textnormal{2}}^{2} + {\bm{w}_{\textnormal{sos}}^{(j)}}^{\top} \bm{\xi}_{\textnormal{sos}}^{(j)} - \bm{\lambda}^{\top} \bm{A}_{\textnormal{sos}}^{(j)}\bm{\xi}_{\textnormal{sos}}^{(j)}, \quad \text{for all } j=1,\dots,m_{\textnormal{sos}}.
    \label{eq:sos-Lagrangian}
\end{align}
\end{subequations}

Consequently, the dual problem of \cref{opt:Reg-QCP} is defined as
\begin{equation}
    \max_{\bm{\lambda}\in\mathbb{R}^{m}} \quad \min_{\bm{\xi} \in K} \mathcal{L}\left(\bm{\xi},\bm{\lambda}\right) \equiv \max_{\bm{\lambda}\in\mathbb{R}^{m}} \quad \min_{\bm{\xi} \in K} \left( \mathcal{L}_{\textnormal{f}}(\bm{\xi}_{\textnormal{f}},\bm{\lambda}) + \sum_{j=1}^{m_{\textnormal{sos}}} \mathcal{L}_{\textnormal{sos}}(\bm{\xi}_{\textnormal{sos}}^{(j)},\bm{\lambda}) \right).
    \nonumber
\end{equation}

In this dual problem, the Lagrangian subproblem consists of minimizing the partial Lagrangian over the primal variables.

Define the dual function $g: \mathbb{R}^{m} \to \mathbb{R}$ as
\begin{equation}
    g(\bm{\lambda}) = \min_{\bm{\xi} \in K} \left( \mathcal{L}_{\textnormal{f}}(\bm{\xi}_{\textnormal{f}},\bm{\lambda}) + \sum_{j=1}^{m_{\textnormal{sos}}} \mathcal{L}_{\textnormal{sos}}(\bm{\xi}_{\textnormal{sos}}^{(j)},\bm{\lambda}) \right),
    \label{eq:dual-function-def}
\end{equation}
then the dual problem of \cref{opt:Reg-QCP} is equivalent to the unconstrained optimization problem
\begin{equation}
    \max_{\bm{\lambda}\in\mathbb{R}^{m}} \quad g(\bm{\lambda}).
    \label{opt:dual-problem-def}
\end{equation}

Next, it is shown that the Lagrangian subproblem has a closed-form solution that allows the dual function to be written explicitly in terms of the dual variables.

\subsection{Closed-form Solutions for the Lagrangian Subproblem}

The Lagrangian subproblem decomposes additively into independent subproblems
\begin{equation}
    \min_{\bm{\xi} \in K} \left( \mathcal{L}_{\textnormal{f}}(\bm{\xi}_{\textnormal{f}},\bm{\lambda}) + \sum_{j=1}^{m_{\textnormal{sos}}} \mathcal{L}_{\textnormal{sos}}(\bm{\xi}_{\textnormal{sos}}^{(j)},\bm{\lambda}) \right) \equiv \min_{\bm{\xi}_{\textnormal{f}} \in \mathbb{R}^{n_{\textnormal{f}}}} \mathcal{L}_{\textnormal{f}}(\bm{\xi}_{\textnormal{f}},\bm{\lambda}) + \sum_{j=1}^{m_{\textnormal{sos}}} \min_{\bm{\xi}_{\textnormal{sos}}^{(j)} \in \mathbb{S}_{+}^{h_j}} \mathcal{L}_{\textnormal{sos}}(\bm{\xi}_{\textnormal{sos}}^{(j)},\bm{\lambda}).
    \nonumber
\end{equation}

For fixed dual variables $\bm{\lambda}$, define the corresponding minimizers of these independent subproblems as
\begin{subequations}
\begin{align}
        \mathcal{S}_{\bm{\xi}_{\textnormal{f}}}(\bm{\lambda}) & \coloneqq \argmin_{\bm{\xi}_{\textnormal{f}} \in \mathbb{R}^{n_{\textnormal{f}}}} \mathcal{L}_{\textnormal{f}}(\bm{\xi}_{\textnormal{f}},\bm{\lambda}),
    \label{eq:ref-sol-free} \\
        \mathcal{S}_{\bm{\xi}_{\textnormal{sos}}^{(j)}}(\bm{\lambda}) & \coloneqq \argmin_{\bm{\xi}_{\textnormal{sos}}^{(j)}\in\mathbb{S}_{+}^{h_{j}}} \mathcal{L}_{\textnormal{sos}}(\bm{\xi}_{\textnormal{sos}}^{(j)},\bm{\lambda}), \quad \text{for all } j=1,\dots,m_{\textnormal{sos}}.
    \label{eq:ref-sol-sos}
\end{align}
\end{subequations}

Each of the independent subproblems admits a closed-form solution, as stated in the following lemmas.

\begin{lemma}
    The minimizer $\mathcal{S}_{\bm{\xi}_{\textnormal{f}}}(\bm{\lambda})$ defined in \cref{eq:ref-sol-free} is given by
    \begin{equation}
            \mathcal{S}_{\bm{\xi}_{\textnormal{f}}}(\bm{\lambda}) = \bm{Q}^{-1} (\bm{A}_{\textnormal{f}}^{\top}\bm{\lambda} - \bm{w}_{\textnormal{f}}).
        \label{eq:free-sol}
    \end{equation}
    \label{lem:free-sol}
\end{lemma}
\begin{proof}
    By completing the square of $\mathcal{L}_{\textnormal{f}}(\bm{\xi}_{\textnormal{f}},\bm{\lambda})$ in \cref{eq:free-Lagrangian}, it holds that
    \begin{equation}
            \mathcal{L}_{\textnormal{f}}(\bm{\xi}_{\textnormal{f}},\bm{\lambda}) = \frac{1}{2}\left\Vert \bm{Q}^{1/2}\bm{\xi}_{\textnormal{f}} - \bm{Q}^{-1/2}\left( \bm{A}_{\textnormal{f}}^{\top}\bm{\lambda} - \bm{w}_{\textnormal{f}}\right)\right\Vert_{2}^{2} - \frac{1}{2}\left\Vert \bm{Q}^{-1/2}\left( \bm{A}_{\textnormal{f}}^{\top}\bm{\lambda} - \bm{w}_{\textnormal{f}}\right)\right\Vert_{2}^{2} + \bm{\lambda}^{\top}\bm{b},
        \label{eq:complete-square-free}
    \end{equation}
    where only the first term involves $\bm{\xi}_{\textnormal{f}}$. Hence, it holds the following equivalence,
    \[
        \argmin_{\bm{\xi}_{\textnormal{f}} \in \mathbb{R}^{n_{\textnormal{f}}}} \mathcal{L}_{\textnormal{f}}(\bm{\xi}_{\textnormal{f}},\bm{\lambda}) = \argmin_{\bm{\xi}_{\textnormal{f}} \in \mathbb{R}^{n_{\textnormal{f}}}} \frac{1}{2}\left\Vert \bm{Q}^{1/2}\bm{\xi}_{\textnormal{f}} - \bm{Q}^{-1/2}\left( \bm{A}_{\textnormal{f}}^{\top}\bm{\lambda} - \bm{w}_{\textnormal{f}}\right)\right\Vert_{2}^{2}. 
    \]
    
    This problem is analogue to a least-squares problem $\min_{\bm{\beta}} \frac{1}{2}\left\Vert \bm{\Theta}\bm{\beta} - \bm{\Psi} \right\Vert_{2}^{2},$ with $\bm{\beta} = \bm{\xi}_{\textnormal{f}}$, $\bm{\Theta} = \bm{Q}^{1/2}$, and $\bm{\Psi} = \bm{Q}^{-1/2}( \bm{A}_{\textnormal{f}}^{\top}\bm{\lambda} - \bm{w}_{\textnormal{f}})$, which has a well-known closed form solution defined by $\bm{\beta}^{\star} = (\bm{\Theta}^{\top}\bm{\Theta})^{-1}\bm{\Theta}^{\top}\bm{\Psi}$. Therefore, the solution to the minimization of $\mathcal{L}_{\textnormal{f}}(\bm{\xi}_{\textnormal{f}},\bm{\lambda})$ over $\bm{\xi}_{\textnormal{f}}$ is given by \cref{eq:free-sol}.
\end{proof}

\begin{lemma}
    The minimizer $\mathcal{S}_{\bm{\xi}_{\textnormal{sos}}^{(j)}}(\bm{\lambda})$ defined in \cref{eq:ref-sol-sos} is given by
    \begin{equation}
            \mathcal{S}_{\bm{\xi}_{\textnormal{sos}}^{(j)}}(\bm{\lambda}) = \mathcal{P}_{\mathbb{S}_{+}^{h_{j}}}\left(\frac{1}{\rho}\left({\bm{A}_{\textnormal{sos}}^{(j)}}^{\top}\bm{\lambda} - \bm{w}_{\textnormal{sos}}^{(j)}\right)\right), \quad \text{for all } j=1,\dots,m_{\textnormal{sos}}.
        \label{eq:sos-sol}
    \end{equation}
    \label{lem:sos-sol}
\end{lemma}
\begin{proof}
    By completing the square of $\mathcal{L}_{\textnormal{sos}}(\bm{\xi}_{\textnormal{sos}}^{(j)},\bm{\lambda})$ in \cref{eq:sos-Lagrangian}, it holds that
    \begin{equation}
        \mathcal{L}_{\textnormal{sos}}(\bm{\xi}_{\textnormal{sos}}^{(j)},\bm{\lambda}) = \frac{\rho}{2}\left\Vert \bm{\xi}_{\textnormal{sos}}^{(j)} - \frac{1}{\rho}\left({\bm{A}_{\textnormal{sos}}^{(j)}}^{\top}\bm{\lambda} - \bm{w}_{\textnormal{sos}}^{(j)}\right)\right\Vert_{2}^{2} - \frac{1}{2}\left\Vert\frac{1}{\rho}\left({\bm{A}_{\textnormal{sos}}^{(j)}}^{\top}\bm{\lambda} - \bm{w}_{\textnormal{sos}}^{(j)}\right)\right\Vert_{2}^{2},
        \label{eq:complete-square-sos}
    \end{equation}
    for all $j=1,\dots,m_{\textnormal{sos}}$, where only the first term involves $\bm{\xi}_{\textnormal{sos}}^{(j)}$. Hence, the following equivalence holds,
    \[
        \argmin_{\bm{\xi}_{\textnormal{sos}}^{(j)}\in\mathbb{S}_{+}^{h_{j}}} \mathcal{L}_{\textnormal{sos}}(\bm{\xi}_{\textnormal{sos}}^{(j)},\bm{\lambda}) =  \argmin_{\bm{\xi}_{\textnormal{sos}}^{(j)}\in\mathbb{S}_{+}^{h_{j}}} \frac{\rho}{2}\left\Vert \bm{\xi}_{\textnormal{sos}}^{(j)} - \frac{1}{\rho}\left({\bm{A}_{\textnormal{sos}}^{(j)}}^{\top}\bm{\lambda} - \bm{w}_{\textnormal{sos}}^{(j)}\right)\right\Vert_{2}^{2}.
    \]
     
    Therefore, the problem of minimizing $\mathcal{L}_{\textnormal{sos}}(\bm{\xi}_{\textnormal{sos}}^{(j)},\bm{\lambda})$ over $\bm{\xi}_{\textnormal{sos}}^{(j)}$ is equivalent to the following optimization problem,
    \[
        \min_{\bm{\xi}_{\textnormal{sos}}^{(j)}\in\mathbb{S}_{+}^{h_{j}}} \quad \frac{\rho}{2}\left\Vert \bm{\xi}_{\textnormal{sos}}^{(j)} - \frac{1}{\rho}\left({\bm{A}_{\textnormal{sos}}^{(j)}}^{\top}\bm{\lambda} - \bm{w}_{\textnormal{sos}}^{(j)}\right)\right\Vert_{2}^{2}.   
    \]

    Since $\bm{\xi}_{\textnormal{sos}}^{(j)} = \operatorname*{svec}(\bm{\Xi}^{(j)})$, by the isometry between the spaces $(\mathbb{R}^{h_{j}(h_{j} + 1)/2}, \Vert \cdot \Vert_{2})$ and $(\mathbb{S}^{h_{j}}, \Vert \cdot \Vert_{\textnormal{F}})$, this is equivalent to $\min_{\bm{\Xi}^{(j)}\in\mathbb{S}_{+}^{h_{j}}} \frac{\rho}{2}\left\Vert \bm{\Xi}^{(j)} - \operatorname*{smat}\left(\frac{1}{\rho}\left({\bm{A}_{\textnormal{sos}}^{(j)}}^{\top}\bm{\lambda} - \bm{w}_{\textnormal{sos}}^{(j)}\right)\right)\right\Vert_{\textnormal{F}}^{2}$, which corresponds to computing a nearest symmetric positive semidefinite matrix. Thus, the minimizer of this problem is obtained by projecting $\operatorname*{smat}({\rho}^{-1}({\bm{A}_{\textnormal{sos}}^{(j)}}^{\top}\bm{\lambda} - \bm{w}_{\textnormal{sos}}^{(j)})) \in \mathbb{S}^{h_{j}}$ onto the positive semidefinite cone $\mathbb{S}_{+}^{h_{j}}$. This projection operation admits a closed-form expression of the form $\mathcal{P}_{\mathbb{S}_{+}^{h}}(\bm{M}) = \bm{V}\bm{D}_{+}\bm{V}^{\top}$, where $\bm{M} = \bm{V}\bm{D}\bm{V}^{\top}$ is the eigenvalue decomposition and $\bm{D}_{+} = \max(\bm{D},\bm{0})$ (see Section 2 in \cite{higham1988computing}, Theorem 2.4 in \cite{malick2004dual}). Then, the solution of minimizing $\mathcal{L}_{\textnormal{sos}}(\bm{\xi}_{\textnormal{sos}}^{(j)},\bm{\lambda})$ over $\bm{\xi}_{\textnormal{sos}}^{(j)}$ for all $j=1,\dots,m_{\textnormal{sos}}$ is given by \cref{eq:sos-sol}.
\end{proof}

We now can simplify the dual function $g$ in \cref{eq:dual-function-def} by removing the minimization through substituting the closed-form expressions for $\mathcal{S}_{\bm{\xi}_{\textnormal{f}}}(\bm{\lambda})$ and $\mathcal{S}_{\bm{\xi}_{\textnormal{sos}}^{(j)}}(\bm{\lambda})$ given in \cref{lem:free-sol} and \cref{lem:sos-sol},
\begin{equation}
\begin{aligned}
    g(\bm{\lambda})
    &= -\frac{1}{2}\bm{\lambda}^{\top}
        \bm{A}_{\textnormal{f}}\bm{Q}^{-1}\bm{A}_{\textnormal{f}}^{\top}\bm{\lambda}
       + \bm{\lambda}^{\top}\bigl(
            \bm{A}_{\textnormal{f}}\bm{Q}^{-1}\bm{w}_{\textnormal{f}}
            + \bm{b}
         \bigr)
       - \frac{1}{2}\bm{w}_{\textnormal{f}}^{\top}\bm{Q}^{-1}\bm{w}_{\textnormal{f}} \\
    &\quad
    + \sum_{j=1}^{m_{\textnormal{sos}}}
        \Biggl(
            \frac{\rho}{2}
            \Bigl\|
                \mathcal{P}_{\mathbb{S}_{+}^{h_{j}}}\Bigl(
                    \frac{1}{\rho}\bigl(
                        {\bm{A}_{\textnormal{sos}}^{(j)}}^{\top}\bm{\lambda}
                        - \bm{w}_{\textnormal{sos}}^{(j)}
                    \bigr)
                \Bigr)
                - \frac{1}{\rho}\bigl(
                    {\bm{A}_{\textnormal{sos}}^{(j)}}^{\top}\bm{\lambda}
                    - \bm{w}_{\textnormal{sos}}^{(j)}
                \bigr)
            \Bigr\|_{2}^{2} \\
    &\qquad\qquad
            - \frac{1}{2}
              \Bigl\|
                  \frac{1}{\rho}\bigl(
                      {\bm{A}_{\textnormal{sos}}^{(j)}}^{\top}\bm{\lambda}
                      - \bm{w}_{\textnormal{sos}}^{(j)}
                  \bigr)
              \Bigr\|_{2}^{2}
        \Biggr).
\end{aligned}
\label{eq:dual-function-complete}
\end{equation}

We next present Moreau's decomposition theorem and a corollary, allowing the latter to be applied into \cref{eq:dual-function-complete} to further simplify the expression of the dual function.

\begin{theorem}[Section III, Theorem 3.2.5 in \cite{hiriart2013convex}, Proposition 1 in \cite{moreau1962}]
    Let $K$ be a closed convex cone in the Hilbert space $(\mathbb{H},\langle\cdot,\cdot\rangle)$ and a closed and convex polar cone $K^{\circ} = \left\{\bm{a} \in \mathbb{H}: \langle\bm{a},\bm{b}\rangle \leq 0, \forall~\bm{b} \in K\right\}$. For $\bm{x},\bm{y},\bm{z} \in \mathbb{H}$, the following statements are equivalent:
    \begin{enumerate}
        \item $\bm{z} = \bm{x} + \bm{y}$, $\bm{x} \in K$, $\bm{y} \in K^{\circ}$, and $\langle\bm{x},\bm{y}\rangle = 0$;
        \item $\bm{x} = \mathcal{P}_{K}(\bm{z})$ and $\bm{y} = \mathcal{P}_{K^{\circ}}(\bm{z})$.
    \end{enumerate}
    \label{thm_Moreau}
\end{theorem}

\begin{corollary}
    Consider the closed convex cone $\mathbb{S}_{+}^{h} \subset \mathbb{S}^{h}$ and $\bm{Z} \in \mathbb{S}^{h}$. The following equality holds,
    \begin{equation}
        \begin{split}
            \left\Vert\mathcal{P}_{\mathbb{S}_{+}^{h}}(\bm{Z}) - \bm{Z} \right\Vert_{\textnormal{F}}^{2} - \left\Vert\bm{Z}\right\Vert_{\textnormal{F}}^{2} = -\left\Vert\mathcal{P}_{\mathbb{S}_{+}^{h}}(\bm{Z})\right\Vert_{\textnormal{F}}^{2}.
        \end{split}
        \label{eq:cor_proj}
    \end{equation}
    \label{cor: Moreau with proj}
\end{corollary}
\begin{proof}
The set $\mathbb{S}^{h}$ equipped with the Frobenius inner product $\langle\bm{A},\bm{B}\rangle_{\textnormal{F}}$ define the Hilbert space $(\mathbb{S}^{h},\langle\cdot,\cdot\rangle_{\textnormal{F}})$. It holds that,
\begin{equation}
\begin{aligned}
    \left\Vert\mathcal{P}_{\mathbb{S}_{+}^{h}}(\bm{Z}) - \bm{Z} \right\Vert_{\mathrm{F}}^{2}
    - \left\Vert\bm{Z}\right\Vert_{\mathrm{F}}^{2}
    &= \left\Vert\mathcal{P}_{\mathbb{S}_{+}^{h}}(\bm{Z})\right\Vert_{\mathrm{F}}^{2}
       - 2\left\langle
            \mathcal{P}_{\mathbb{S}_{+}^{h}}(\bm{Z}),
            \bm{Z}
         \right\rangle_{\mathrm{F}} \\
    &= \left\Vert\mathcal{P}_{\mathbb{S}_{+}^{h}}(\bm{Z})\right\Vert_{\mathrm{F}}^{2}
       - 2\left\langle
            \mathcal{P}_{\mathbb{S}_{+}^{h}}(\bm{Z}),
            \mathcal{P}_{\mathbb{S}_{+}^{h}}(\bm{Z})
         \right\rangle_{\mathrm{F}} \\
    &\quad
       - 2\left\langle
            \mathcal{P}_{\mathbb{S}_{+}^{h}}(\bm{Z}),
            \mathcal{P}_{\mathbb{S}_{+}^{o,h}}(\bm{Z})
         \right\rangle_{\mathrm{F}} \\
    &= -\left\Vert\mathcal{P}_{\mathbb{S}_{+}^{h}}(\bm{Z})\right\Vert_{\mathrm{F}}^{2}
       - 2\left\langle
            \mathcal{P}_{\mathbb{S}_{+}^{h}}(\bm{Z}),
            \mathcal{P}_{\mathbb{S}_{+}^{o,h}}(\bm{Z})
         \right\rangle_{\mathrm{F}}.
\end{aligned}
\end{equation}

Since from \cref{thm_Moreau}, $\langle\mathcal{P}_{\mathbb{S}_{+}^{h}}(\bm{Z}),\mathcal{P}_{\mathbb{S}_{+}^{o,h}}(\bm{Z})\rangle_{\textnormal{F}} = 0$, it follows that \cref{eq:cor_proj} holds. Moreover, by the isometry between the spaces $(\mathbb{R}^{h(h + 1)/2}, \Vert \cdot \Vert_{2})$ and $(\mathbb{S}^{h}, \Vert \cdot \Vert_{\textnormal{F}})$, it holds that
\begin{equation}
    \left\Vert\mathcal{P}_{\mathbb{S}_{+}^{h}}(\bm{z}) - \bm{z} \right\Vert_{2}^{2} - \left\Vert\bm{z}\right\Vert_{2}^{2} = -\left\Vert\mathcal{P}_{\mathbb{S}_{+}^{h}}(\bm{z})\right\Vert_{2}^{2}.
    \nonumber
\end{equation}
\end{proof}

Therefore, the results of \cref{cor: Moreau with proj} establish the following equivalent form of the dual function
\begin{equation}
    \begin{split}
        g(\bm{\lambda}) =~& -\frac{1}{2}\bm{\lambda}^{\top}\bm{A}_{\textnormal{f}}\bm{Q}^{-1} \bm{A}_{\textnormal{f}}^{\top}\bm{\lambda} + \bm{\lambda}^{\top}\left( \bm{A}_{\textnormal{f}}\bm{Q}^{-1}\bm{w}_{\textnormal{f}} + \bm{b}\right)- \frac{1}{2}\bm{w}_{\textnormal{f}}^{\top}\bm{Q}^{-1}\bm{w}_{\textnormal{f}} -\\
        & \frac{\rho}{2} \sum_{j=1}^{m_{\textnormal{sos}}}  \left\Vert\mathcal{P}_{\mathbb{S}_{+}^{h_{j}}}\left(\frac{1}{\rho}\left({\bm{A}_{\textnormal{sos}}^{(j)}}^{\top}\bm{\lambda} - \bm{w}_{\textnormal{sos}}^{(j)}\right)\right)\right\Vert_{2}^{2},
    \end{split}
    \label{eq:dual-function-lambda}
\end{equation}
which depends solely on the dual variable $\bm{\lambda}$.

\subsection{Smoothness of the Dual Problem}

The following lemma proves concavity and differentiability of the dual function $g$ in \cref{eq:dual-function-lambda}.

\begin{lemma}
    The dual function $g$, defined in \cref{eq:dual-function-lambda}, is concave and differentiable on $\mathbb{R}^{m}$.
    \label{lem: dual concave and differentiable}
\end{lemma}
\begin{proof}
    Note that the separability of $g$ in \cref{eq:dual-function-lambda} allows for a sequential analysis of the concavity and differentiability of each term. In fact, it holds that the expression for $g$ comprises a subtraction of a quadratic form in $\bm{\lambda}$, a linear function in $\bm{\lambda}$, a constant, and a sum of norms of projections onto positive semidefinite cones. Each of these components is analyzed separately.

    The term $-\frac{1}{2}\bm{\lambda}^{\top}\bm{A}_{\textnormal{f}}\bm{Q}^{-1} \bm{A}_{\textnormal{f}}^{\top}\bm{\lambda}$ is concave if the matrix $\bm{A}_{\textnormal{f}}\bm{Q}^{-1} \bm{A}_{\textnormal{f}}^{\top}$ is positive semidefinite. Because $\bm{Q} \in \mathbb{S}_{\succ 0}^{n_{\textnormal{f}}}$, its inverse is also positive definite and admits a decomposition $\bm{Q}^{-1} = \bm{B}^{\top}\bm{B}$. Let $\hat{\bm{B}} = \bm{B}\bm{A}_{\textnormal{f}}^{\top}$, then $\bm{A}_{\textnormal{f}}\bm{Q}^{-1} \bm{A}_{\textnormal{f}}^{\top} = \hat{\bm{B}}^{\top}\hat{\bm{B}} \in \mathbb{S}_{+}^{m}$. Hence, the quadratic part is concave.

    The linear term $\bm{\lambda}^{\top}\left( \bm{A}_{\textnormal{f}}\bm{Q}^{-1}\bm{w}_{\textnormal{f}} + \bm{b}\right)$ is both concave and convex, and the constant term $-\frac{1}{2}\bm{w}_{\textnormal{f}}^{\top}\bm{Q}^{-1}\bm{w}_{\textnormal{f}}$ does not affect concavity.  Moreover, the quadratic, linear, and constant terms are clearly differentiable on $\mathbb{R}^{m}$.

    Next, for each $j = 1,\dots,m_{\textnormal{sos}}$, the function $\frac{\rho}{2} \left\Vert\mathcal{P}_{\mathbb{S}_{+}^{h_{j}}}\left(\frac{1}{\rho}\left({\bm{A}_{\textnormal{sos}}^{(j)}}^{\top}\bm{\lambda} - \bm{w}_{\textnormal{sos}}^{(j)}\right)\right)\right\Vert_{2}^{2}$ is the squared norm of a projection, which is known to be convex and differentiable (see Theorem 3.1 in \cite{malick2004dual}). Consequently, its negative $- \frac{\rho}{2} \left\Vert\mathcal{P}_{\mathbb{S}_{+}^{h_{j}}}\left(\frac{1}{\rho}\left({\bm{A}_{\textnormal{sos}}^{(j)}}^{\top}\bm{\lambda} - \bm{w}_{\textnormal{sos}}^{(j)}\right)\right)\right\Vert_{2}^{2}$ is concave and differentiable.

    Since a sum of concave functions is concave, and each of the four components above is concave, it follows that $g$ is concave on $\mathbb{R}^{m}$. The same argument proves that $g$ is differentiable on $\mathbb{R}^{m}$.
\end{proof}

The dual function $g$ in \cref{eq:dual-function-lambda} is $L$-smooth on $\mathbb{R}^{m}$, as proved in the following theorem.

\begin{theorem}
    The gradient $\nabla g: \mathbb{R}^{m} \to \mathbb{R}^{m}$ of the dual function $g$ in \cref{eq:dual-function-lambda} is defined by
    \begin{equation}
        \nabla g(\bm{\lambda}) = -\bm{A}_{\textnormal{f}}\bm{Q}^{-1} \bm{A}_{\textnormal{f}}^{\top}\bm{\lambda} + \bm{A}_{\textnormal{f}}\bm{Q}^{-1}\bm{w}_{\textnormal{f}} + \bm{b} - \sum_{j=1}^{m_{\textnormal{sos}}} \bm{A}_{\textnormal{sos}}^{(j)} \mathcal{P}_{\mathbb{S}_{+}^{h_{j}}}\left(\frac{1}{\rho}\left({\bm{A}_{\textnormal{sos}}^{(j)}}^{\top}\bm{\lambda} - \bm{w}_{\textnormal{sos}}^{(j)}\right)\right).
        \label{eq:gradient-dual}
    \end{equation}
    
    The gradient $\nabla g$ is Lipschitz continuous with constant $L$ given by

    \begin{equation}
        \begin{split}
            L = \left\Vert \bm{A}_{\textnormal{f}}\bm{Q}^{-1} \bm{A}_{\textnormal{f}}^{\top}\right\Vert_{2} + \frac{1}{\rho}  \left\Vert \sum_{j=1}^{m_{\textnormal{sos}}} \bm{A}_{\textnormal{sos}}^{(j)} {\bm{A}_{\textnormal{sos}}^{(j)}}^{\top} \right\Vert_{2}.
        \end{split}
        \label{eq:lipschitz-constant}
    \end{equation}
    
    \label{thm:gradient-lipschitz}
\end{theorem}
\begin{proof}
    By linearity of the differentiation, it holds that
    \begin{equation}
        \nabla g(\bm{\lambda}) = -\bm{A}_{\textnormal{f}}\bm{Q}^{-1} \bm{A}_{\textnormal{f}}^{\top}\bm{\lambda} + \bm{A}_{\textnormal{f}}\bm{Q}^{-1}\bm{w}_{\textnormal{f}} + \bm{b} - \frac{\rho}{2} \sum_{j=1}^{m_{\textnormal{sos}}} \frac{\partial}{\partial \bm{\lambda}}  \left\Vert\mathcal{P}_{\mathbb{S}_{+}^{h_{j}}}\left(\frac{1}{\rho}\left({\bm{A}_{\textnormal{sos}}^{(j)}}^{\top}\bm{\lambda} - \bm{w}_{\textnormal{sos}}^{(j)}\right)\right)\right\Vert_{2}^{2}.
        \label{eq:gradient-aux}
    \end{equation}

    Theorems 2.2 and 3.2 in \cite{malick2004dual} proved that
    \begin{equation}
        - \frac{1}{2} \frac{\partial}{\partial \bm{y}}  \left\Vert\mathcal{P}_{\mathcal{K}}\left(\bm{c} + \mathcal{A}^{*} \bm{y}\right)\right\Vert^{2} = - \mathcal{A}\mathcal{P}_{\mathcal{K}}\left(\bm{c} + \mathcal{A}^{*} \bm{y}\right).
    \end{equation}

    Set $\bm{y} = \bm{\lambda}$, $\mathcal{A} = \frac{1}{\rho}\bm{A}_{\textnormal{sos}}^{(j)}$, $\bm{c} = -\frac{1}{\rho}\bm{w}_{\textnormal{sos}}^{(j)}$, and $\mathcal{K} = \mathbb{S}_{+}^{h_{j}}$ for all $j=1,\dots,m_{\textnormal{sos}}$, it follows that
    \begin{equation}
    \begin{aligned}
        - \frac{\rho}{2}
          \sum_{j=1}^{m_{\textnormal{sos}}}
          \frac{\partial}{\partial \bm{\lambda}}
          \left\Vert
              \mathcal{P}_{\mathbb{S}_{+}^{h_{j}}}\!\left(
                  \frac{1}{\rho}
                  \left(
                      {\bm{A}_{\textnormal{sos}}^{(j)}}^{\top}\bm{\lambda}
                      - \bm{w}_{\textnormal{sos}}^{(j)}
                  \right)
              \right)
          \right\Vert_{2}^{2}
        \\
        =\;
        - \sum_{j=1}^{m_{\textnormal{sos}}}
            \bm{A}_{\textnormal{sos}}^{(j)}
            \mathcal{P}_{\mathbb{S}_{+}^{h_{j}}}\!\left(
                \frac{1}{\rho}
                \left(
                    {\bm{A}_{\textnormal{sos}}^{(j)}}^{\top}\bm{\lambda}
                    - \bm{w}_{\textnormal{sos}}^{(j)}
                \right)
            \right).
    \end{aligned}
    \end{equation}
    Substituting this into \cref{eq:gradient-aux} yields \cref{eq:gradient-dual} for the gradient of the dual function.

    Next, we want to show that the gradient of the dual function $\nabla g(\bm{\lambda})$ is Lipschitz-continuous, and to compute a tight Lipschitz constant. The gradient expression in \cref{eq:gradient-dual} has two parts: a linear term coming from the free variables and a projection term coming from the conic constrained variables. We treat each part separately.

    The free-variable contribution to the gradient is $-\bm{A}_{\textnormal{f}}\bm{Q}^{-1} \bm{A}_{\textnormal{f}}^{\top}\bm{\lambda}$. Then, for any $\bm{\alpha}, \bm{\beta} \in \mathbb{R}^{m}$, by Cauchy-Schwarz,
    \begin{equation}
        \left\Vert -\bm{A}_{\textnormal{f}}\bm{Q}^{-1} \bm{A}_{\textnormal{f}}^{\top}\bm{\alpha} + \bm{A}_{\textnormal{f}}\bm{Q}^{-1} \bm{A}_{\textnormal{f}}^{\top}\bm{\beta} \right\Vert_{2} \leq \left\Vert \bm{A}_{\textnormal{f}}\bm{Q}^{-1} \bm{A}_{\textnormal{f}}^{\top} \right\Vert_{2} \left\Vert \bm{\alpha} - \bm{\beta} \right\Vert_{2}.
        \label{eq:lipschitz-linear}
    \end{equation}
    
    The projection of symmetric matrices onto the cone $\mathbb{S}_{+}^{h}$ is nonexpansive in the Frobenius norm and, due to the isometry between the spaces $(\mathbb{R}^{h(h + 1)/2}, \Vert \cdot \Vert_{2})$ and $(\mathbb{S}^{h}, \Vert \cdot \Vert_{\textnormal{F}})$, it holds that
    \begin{equation}
    \begin{aligned}
        \left\Vert
            \mathcal{P}_{\mathbb{S}_{+}^{h_{j}}}\!\left(
                \frac{1}{\rho}
                \left(
                    {\bm{A}_{\textnormal{sos}}^{(j)}}^{\top}\bm{\alpha}
                    - \bm{w}_{\textnormal{sos}}^{(j)}
                \right)
            \right)
            -
            \mathcal{P}_{\mathbb{S}_{+}^{h_{j}}}\!\left(
                \frac{1}{\rho}
                \left(
                    {\bm{A}_{\textnormal{sos}}^{(j)}}^{\top}\bm{\beta}
                    - \bm{w}_{\textnormal{sos}}^{(j)}
                \right)
            \right)
        \right\Vert_{2}
        \\
        \leq\;
        \left\Vert
            \frac{1}{\rho}\,
            {\bm{A}_{\textnormal{sos}}^{(j)}}^{\top}
            (\bm{\alpha} - \bm{\beta})
        \right\Vert_{2}.
    \end{aligned}
    \label{eq:nonexp}
    \end{equation}
    
    Now, define $\Delta \in \mathbb{R}^{m}$ as follows
    \[
        \Delta \coloneqq \sum_{j=1}^{m_{\textnormal{sos}}} \bm{A}_{\textnormal{sos}}^{(j)} \left [ \mathcal{P}_{\mathbb{S}_{+}^{h_{j}}}\left(\frac{1}{\rho}\left({\bm{A}_{\textnormal{sos}}^{(j)}}^{\top}\bm{\alpha} - \bm{w}_{\textnormal{sos}}^{(j)}\right)\right) -  \mathcal{P}_{\mathbb{S}_{+}^{h_{j}}}\left(\frac{1}{\rho}\left({\bm{A}_{\textnormal{sos}}^{(j)}}^{\top}\bm{\beta} - \bm{w}_{\textnormal{sos}}^{(j)}\right)\right) \right ],
    \]
    and $\Gamma \coloneqq \sum_{j=1}^{m_{\textnormal{sos}}} \bm{A}_{\textnormal{sos}}^{(j)} {\bm{A}_{\textnormal{sos}}^{(j)}}^{\top}$. The Euclidean norm of $\Delta$ is $\| \Delta \|_{2} = \max_{\| \bm{z} \|_{2} = 1} {\bm{z}}^{\top} \Delta$, with $\bm{z} \in \mathbb{R}^{m}$. For any $\| \bm{z} \|_{2} = 1$ it follows that,
    \begin{equation}
    \begin{aligned}
        &{\bm{z}}^{\top} \Delta
        = \sum_{j=1}^{m_{\textnormal{sos}}}
           {\left({\bm{A}_{\textnormal{sos}}^{(j)}}^{\top} \bm{z}\right)}^{\top}
           \Bigl[
               \mathcal{P}_{\mathbb{S}_{+}^{h_{j}}}\!\left(
                   \frac{1}{\rho}
                   \left(
                       {\bm{A}_{\textnormal{sos}}^{(j)}}^{\top}\bm{\alpha}
                       - \bm{w}_{\textnormal{sos}}^{(j)}
                   \right)
               \right)
               -
               \mathcal{P}_{\mathbb{S}_{+}^{h_{j}}}\!\left(
                   \frac{1}{\rho}
                   \left(
                       {\bm{A}_{\textnormal{sos}}^{(j)}}^{\top}\bm{\beta}
                       - \bm{w}_{\textnormal{sos}}^{(j)}
                   \right)
               \right)
           \Bigr] \\
        &\stackrel{\text{C.S.}}{\leq}
           \sum_{j=1}^{m_{\textnormal{sos}}}
           \left\|{\bm{A}_{\textnormal{sos}}^{(j)}}^{\top} \bm{z}\right\|_{2}
           \left\|
               \mathcal{P}_{\mathbb{S}_{+}^{h_{j}}}\!\left(
                   \frac{1}{\rho}
                   \left(
                       {\bm{A}_{\textnormal{sos}}^{(j)}}^{\top}\bm{\alpha}
                       - \bm{w}_{\textnormal{sos}}^{(j)}
                   \right)
               \right)
               -
               \mathcal{P}_{\mathbb{S}_{+}^{h_{j}}}\!\left(
                   \frac{1}{\rho}
                   \left(
                       {\bm{A}_{\textnormal{sos}}^{(j)}}^{\top}\bm{\beta}
                       - \bm{w}_{\textnormal{sos}}^{(j)}
                   \right)
               \right)
           \right\|_{2} \\
        &\stackrel{\cref{eq:nonexp}}{\leq}
           \frac{1}{\rho}
           \sum_{j=1}^{m_{\textnormal{sos}}}
           \left\|{\bm{A}_{\textnormal{sos}}^{(j)}}^{\top} \bm{z}\right\|_{2}
           \left\|
               {\bm{A}_{\textnormal{sos}}^{(j)}}^{\top}(\bm{\alpha} - \bm{\beta})
           \right\|_{2} \\
        &\stackrel{\text{C.S.}}{\leq}
           \frac{1}{\rho}
           \left(
               \sum_{j=1}^{m_{\textnormal{sos}}}
               \left\|{\bm{A}_{\textnormal{sos}}^{(j)}}^{\top} \bm{z}\right\|_{2}^{2}
           \right)^{1/2}
           \left(
               \sum_{j=1}^{m_{\textnormal{sos}}}
               \left\|
                   {\bm{A}_{\textnormal{sos}}^{(j)}}^{\top}(\bm{\alpha} - \bm{\beta})
               \right\|_{2}^{2}
           \right)^{1/2} \\
        &= \frac{1}{\rho}
           \left( {\bm{z}}^{\top} \Gamma \bm{z} \right)^{1/2}
           \left( {(\bm{\alpha} - \bm{\beta})}^{\top} \Gamma (\bm{\alpha} - \bm{\beta}) \right)^{1/2} \\
        &\leq
           \frac{1}{\rho}
           \|\Gamma\|_{2}^{1/2}
           \|\Gamma\|_{2}^{1/2}
           \|\bm{\alpha} - \bm{\beta}\|_{2} = \frac{1}{\rho}\,\|\Gamma\|_{2}\,\|\bm{\alpha} - \bm{\beta}\|_{2},
    \end{aligned}
    \end{equation}
    where C.S is the Cauchy-Schwarz inequality and the last inequality comes from ${\bm{z}}^{\top}\Gamma\bm{z}\le\|\Gamma\|_{2}\|\bm{z}\|_{2}^{2}$ for $\| \bm{z} \|_{2} = 1$.
    
    This inequality implies that,
    \begin{equation}
        \| \Delta \|_{2} \leq \frac{1}{\rho} \| \Gamma \|_{2} \| \bm{\alpha} - \bm{\beta} \|_{2} = \frac{1}{\rho} \left\Vert \sum_{j=1}^{m_{\textnormal{sos}}} \bm{A}_{\textnormal{sos}}^{(j)} {\bm{A}_{\textnormal{sos}}^{(j)}}^{\top} \right\Vert_{2} \left\Vert \bm{\alpha} - \bm{\beta} \right\Vert_{2},
        \label{eq:lipschitz-proj}
    \end{equation}
    which, by the definition of $\Delta$, defines the Lipschitz constant of the sum of projections in \cref{eq:gradient-dual}.
    
    By combining the bounds in \cref{eq:lipschitz-linear} and \cref{eq:lipschitz-proj}, it holds
    \begin{equation}
        \left\Vert \nabla g(\bm{\alpha}) - \nabla g(\bm{\beta}) \right\Vert_{2} \leq \left( \left\Vert \bm{A}_{\textnormal{f}}\bm{Q}^{-1} \bm{A}_{\textnormal{f}}^{\top} \right\Vert_{2} + \frac{1}{\rho} \left\Vert \sum_{j=1}^{m_{\textnormal{sos}}} \bm{A}_{\textnormal{sos}}^{(j)} {\bm{A}_{\textnormal{sos}}^{(j)}}^{\top} \right\Vert_{2} \right) \left\Vert \bm{\alpha} - \bm{\beta} \right\Vert_{2}
    \end{equation}
    which results in \cref{eq:lipschitz-constant} for the Lipschitz constant $L$ of the gradient.
\end{proof}

Because the dual function $g$ is concave and differentiable with Lipschitz gradient, the unconstrained maximization problem in Opt. \cref{opt:dual-problem-def} can be solved efficiently using \cref{alg:first-order-solver}. In this algorithm, we employ an accelerated gradient scheme on this maximization problem and, once an optimal (or nearly optimal) dual variable is obtained, the corresponding primal solution of \cref{opt:Reg-QCP} is recovered via the closed-form expressions in \cref{eq:free-sol} and \cref{eq:sos-sol}.

The next section provides the convergence of \cref{alg:first-order-solver} with respect to how the solutions of \cref{opt:Reg-QCP} are approached. In particular, it is shown that the dual iterates converge at the optimal rate, the primal variables become asymptotically feasible, and the objective value of the regularized problem \cref{opt:Reg-QCP} converges to that of the original problem of \cref{opt:QCP}.

\section{Convergence Analysis}\label{sec:convergence}

The accelerated gradient method in \cref{alg:first-order-solver} is selected specifically for its optimal convergence rate among first-order methods. The following lemma provides this rate, demonstrating how fast the dual objective value approaches its optimum.

\begin{theorem}[Convergence of the Accelerated Gradient Scheme]
    Consider the concave dual function $g$ in \cref{eq:dual-function-lambda} whose gradient is $L$-Lipschitz continuous. Let $\bm{\lambda}^\star$ be a maximizer of $g$. For any initial point $\bm{\lambda}^{(0)} \in \mathbb{R}^m$, step size $\eta \in (0, 1/L]$, and initialization
    \[
    t^{(0)} = 1, \qquad \bm{y}^{(0)} = \bm{\lambda}^{(0)},
    \]
    the iterates generated by the accelerated gradient ascent scheme
    \begin{equation}
        \begin{split}
            t^{(k)} &= \frac{1}{2} + \sqrt{\frac{1}{4} + \bigl(t^{(k-1)}\bigr)^2}, \quad \bm{\lambda}^{(k)} = \bm{y}^{(k-1)} + \eta \nabla g(\bm{y}^{(k-1)}), \\
            \bm{y}^{(k)} &= \bm{\lambda}^{(k)} + \frac{t^{(k-1)} - 1}{t^{(k)}} \bigl( \bm{\lambda}^{(k)} - \bm{\lambda}^{(k-1)} \bigr),
        \end{split}
        \nonumber
    \end{equation}
    for $k = 1, 2, \dots$ satisfy the convergence bound after $N$ iterations
    \[
        g(\bm{\lambda}^\star) - g(\bm{\lambda}^{(N)}) \leq \frac{2L \|\bm{\lambda}^{(0)} - \bm{\lambda}^\star\|_2^2}{(N+1)^2},
    \]
    provided $\eta = 1/L$.
    \label{thm:gradient-rate}
\end{theorem}
\begin{proof}
    The proof follows similar arguments from the classical accelerated scheme in \cite{nesterov1983method}, adapted for the maximization of smooth concave functions.
\end{proof}

From this convergence theorem, \cref{alg:first-order-solver} inherits the convergence rate of $\mathcal{O}(1/N^{2})$ for solving the unconstrained concave dual problem in Opt. \cref{opt:dual-problem-def}, when selecting a step-size $\eta = 1/L$, where $L$ is the gradient's Lipschitz constant in \cref{eq:lipschitz-constant}.

The gradient in \cref{eq:gradient-dual} can be written using the closed-form primal solutions in \cref{eq:free-sol} and \cref{eq:sos-sol},
\begin{equation}
    \nabla g(\bm{\lambda}) = - \left( \bm{A}_{\textnormal{f}}\mathcal{S}_{\bm{\xi}_{\textnormal{f}}}(\bm{\lambda}) + \sum_{j=1}^{m_{\textnormal{sos}}} \bm{A}_{\textnormal{sos}}^{(j)} \mathcal{S}_{\bm{\xi}_{\textnormal{sos}}^{(j)}}(\bm{\lambda}) - \bm{b}\right),
    \label{eq:gradient-dual-feasibility}
\end{equation}
which is exactly the negative of the affine constraint evaluated at primal solutions. Then, the convergence of the dual objective value established in \cref{thm:gradient-rate} has a direct and significant consequence for the recovered primal solutions, since they are recovered once the dual problem is solved. In \cref{alg:first-order-solver}, the norm of the gradient of the dual function serves as a measure of primal feasibility with respect to the affine constraints. In the following theorem, it is proved that the norm of the gradient, and consequently the affine constraint violation, vanishes as the number of iterations increases.

\begin{theorem}[Asymptotic Feasibility of Primal Iterates]
    As the number of dual optimization iterations increases, the primal iterates achieve feasibility for \cref{opt:Reg-QCP}. This progress is measured by the norm of the dual function's gradient given in \cref{eq:gradient-dual-feasibility}, which drives the constraint violation to zero
    \begin{equation}
        \left \Vert \bm{A}_{\textnormal{f}}\mathcal{S}_{\bm{\xi}_{\textnormal{f}}}(\bm{\lambda}^{(N)}) + \sum_{j=1}^{m_{\textnormal{sos}}} \bm{A}_{\textnormal{sos}}^{(j)} \mathcal{S}_{\bm{\xi}_{\textnormal{sos}}^{(j)}}(\bm{\lambda}^{(N)}) - \bm{b} \right \Vert_{2} \to 0 \quad \text{as} \quad N \to \infty.
    \end{equation}
    \label{thm:feasibility}
\end{theorem}
\begin{proof}
    For a concave function $g$ with an $L$-Lipschitz continuous gradient, the difference between the optimal value $g(\bm{\lambda}^{\star})$ and the value at any point $\bm{\lambda}^{(N)}$ is bounded below by the squared norm of the gradient at that point
    \[
        \frac{1}{2L}\Vert \nabla g(\bm{\lambda}^{(N)}) \Vert_{2}^{2} \leq  g(\bm{\lambda}^{\star}) -  g(\bm{\lambda}^{(N)}).
    \]

    Combining this inequality with the bound in \cref{thm:gradient-rate} yields
    \begin{equation}
        \Vert \nabla g(\bm{\lambda}^{(N)}) \Vert_{2} \leq \frac{2L \|\bm{\lambda}^{(0)} - \bm{\lambda}^\star\|_2}{(N + 1)}.
        \label{eq:convergence-feas}
    \end{equation}
    
    As $N \to \infty$, the right-hand side vanishes, which precisely corresponds to the feasibility of the primal iterates with respect to the affine constraints of \cref{opt:Reg-QCP} defined in \cref{eq:gradient-dual-feasibility}.
\end{proof}

We next prove that the dual function $g$ in \cref{eq:dual-function-lambda} is coercive, which is a powerful property that can be further used to establish boundedness for dual iterates in \cref{alg:first-order-solver}. The latter allows us to derive a non-asymptotic bound for suboptimality in terms of how close the objective of \cref{opt:Reg-QCP} is from the optimum value of the original problem of \cref{opt:QCP}.

\begin{theorem}[Coercivity of the Dual Function]
    Assume that the Slater condition holds for \cref{opt:Reg-QCP}; i.e., there exists a primal point $(\bar{\bm{\xi}}_{\textnormal{f}},\bar{\bm{\xi}}_{\textnormal{sos}}^{(1)},\dots,\bar{\bm{\xi}}_{\textnormal{sos}}^{(m_{\textnormal{sos}})})$ such that
    \begin{equation}
        \bm{A}_{\textnormal{f}}\bar{\bm{\xi}}_{\textnormal{f}} + \sum_{j=1}^{m_{\textnormal{sos}}}\bm{A}_{\textnormal{sos}}^{(j)}\bar{\bm{\xi}}_{\textnormal{sos}}^{(j)} = \bm{b}, \quad \bar{\bm{\xi}}_{\textnormal{sos}}^{(j)} \in \mathbb{S}_{\succ 0}^{h_{j}}, \quad \text{for all } j=1,\dots,m_{\textnormal{sos}}.
        \label{eq:slater}
    \end{equation}
    Assume also that the matrices
    \begin{equation}
        \bm{Z}^{(j)} = \bm{A}_{\textnormal{sos}}^{(j)} {\bm{A}_{\textnormal{sos}}^{(j)}}^{\top} \in \mathbb{S}^{m}, \quad \text{for all } j=1,\dots,m_{\textnormal{sos}},
    \end{equation}
    are each invertible. Then the dual function $g$ in \cref{eq:dual-function-lambda} is coercive, satisfying
    \begin{equation}
        \begin{split}
             g(\bm{\lambda}) \leq -\beta_{1}\Vert \bm{\lambda} \Vert_{2} + \beta_{2}, \quad \text{for all } \bm{\lambda} \in \mathbb{R}^{m},
        \end{split}
        \label{eq:coercivity}
    \end{equation}
    where $\beta_{1} > 0$ and $\beta_{2} > 0$ are given by
    \begin{equation}
        \beta_1 = \frac{1}{2}\sum_{j=1}^{m_{\textnormal{sos}}} \sigma_{\textnormal{min}}(\operatorname*{smat}(\bar{\bm{\xi}}_{\textnormal{sos}}^{(j)})) {\sqrt{\sigma_{\textnormal{min}}(\bm{Z}^{(j)}})},
        \label{eq:beta1}
    \end{equation}
    \begin{equation}
    \begin{aligned}
        \beta_2
        &= \frac{1}{2}
           \left\|
               \bm{Q}^{1/2}\bar{\bm{\xi}}_{\textnormal{f}}
               + \bm{Q}^{-1/2}\bm{w}_{\textnormal{f}}
           \right\|_2^{2} 
           + \frac{\rho}{2}
             \sum_{j=1}^{m_{\textnormal{sos}}}
             \left(
                 \left\|
                     \bar{\bm{\xi}}_{\textnormal{sos}}^{(j)}
                 \right\|_{2}
                 + \frac{1}{2}\,
                   \sigma_{\textnormal{min}}\!\left(
                       \operatorname*{smat}\bigl(
                           \bar{\bm{\xi}}_{\textnormal{sos}}^{(j)}
                       \bigr)
                   \right)
             \right)^{2} \\
        &\quad
           + \sum_{j=1}^{m_{\textnormal{sos}}}
             {\bm{w}_{\textnormal{sos}}^{(j)}}^{\top}
             \bar{\bm{\xi}}_{\textnormal{sos}}^{(j)}.
    \end{aligned}
    \label{eq:beta2}
    \end{equation}

    \label{thm:coercivity}
\end{theorem}
\begin{proof}
    Since $g(\bm{\lambda}) \leq \mathcal{L}(\bm{\xi}, \bm{\lambda})$ holds for any $\bm{\xi} \in K$, it suffices to construct a $\bm{\lambda}$-dependent primal point $\hat{\bm{\xi}} \in K$ such that $\mathcal{L}(\hat{\bm{\xi}}, \bm{\lambda})$ satisfies the bound in \cref{eq:coercivity}.

    Let $\delta^{(j)} = \frac{1}{2} \sigma_{\textnormal{min}}(\operatorname*{smat}(\bar{\bm{\xi}}_{\textnormal{sos}}^{(j)})) {\sqrt{\sigma_{\textnormal{min}}(\bm{Z}^{(j)}})} > 0$, where positivity follows from the Slater condition in \cref{eq:slater}. Define $\hat{\bm{\xi}}_{\textnormal{f}} = \bar{\bm{\xi}}_{\textnormal{f}}$ and the perturbed variables
    \begin{equation}
        \hat{\bm{\xi}}_{\textnormal{sos}}^{(j)} = \bar{\bm{\xi}}_{\textnormal{sos}}^{(j)} + \frac{\delta^{(j)}}{\|\bm{\lambda}\|_{2}} {\bm{A}_{\textnormal{sos}}^{(j)}}^{\top} {\bm{Z}^{(j)}}^{-1} \bm{\lambda}, \quad \text{for all } j=1,\dots,m_{\textnormal{sos}}.
        \nonumber
    \end{equation}

    It is now shown that $\|\hat{\bm{\xi}}_{\textnormal{sos}}^{(j)}\|_2$ is bounded by a constant independent of $\bm{\lambda}$ and that $\hat{\bm{\xi}}_{\textnormal{sos}}^{(j)} \in \mathbb{S}^{h_j}_{\succ 0}$, so that $\hat{\bm{\xi}} \in K$ is a feasible point.
    
    By the triangle inequality and Cauchy-Schwarz,
    \begin{equation}
    \begin{aligned}
        \|\hat{\bm{\xi}}_{\textnormal{sos}}^{(j)}\|_2
        &\leq
        \|\bar{\bm{\xi}}_{\textnormal{sos}}^{(j)}\|_2
        + \frac{\delta^{(j)}}{\|\bm{\lambda}\|_2}
          \left\|
              {\bm{A}_{\textnormal{sos}}^{(j)}}^{\top}
              {\bm{Z}^{(j)}}^{-1}
              \bm{\lambda}
          \right\|_2
        \\
        &= 
        \|\bar{\bm{\xi}}_{\textnormal{sos}}^{(j)}\|_2
        + \frac{\delta^{(j)}}{\|\bm{\lambda}\|_2}
          \sqrt{
              \bm{\lambda}^{\top}
              {\bm{Z}^{(j)}}^{-1}
              \bm{\lambda}
          }
        \\
        &\leq
        \|\bar{\bm{\xi}}_{\textnormal{sos}}^{(j)}\|_2
        + \frac{\delta^{(j)}}{
            \sqrt{
                \sigma_{\textnormal{min}}\!\left(
                    \bm{Z}^{(j)}
                \right)
            }
          }.
    \end{aligned}
    \end{equation}

    where the last inequality follows from $\bm{\lambda}^\top {\bm{Z}^{(j)}}^{-1}\bm{\lambda} \leq \|\bm{\lambda}\|_2^2/\sigma_{\textnormal{min}}(\bm{Z}^{(j)})$. Substituting $\delta^{(j)}$ gives
    \begin{equation}
        \|\hat{\bm{\xi}}_{\textnormal{sos}}^{(j)}\|_{2} \leq \|\bar{\bm{\xi}}_{\textnormal{sos}}^{(j)}\|_{2} + \frac{1}{2} \sigma_{\textnormal{min}}(\operatorname*{smat}(\bar{\bm{\xi}}_{\textnormal{sos}}^{(j)})).
        \nonumber
    \end{equation}
    which is a constant independent of $\bm{\lambda}$.

    For any $\bm{v} \in \mathbb{R}^{h_j}$, decompose the quadratic form as
    \begin{equation}
        \bm{v}^\top \operatorname*{smat}(\hat{\bm{\xi}}_{\textnormal{sos}}^{(j)})\bm{v} = \bm{v}^\top \operatorname*{smat}(\bar{\bm{\xi}}_{\textnormal{sos}}^{(j)})\bm{v} 
        + \frac{\delta^{(j)}}{\|\bm{\lambda}\|_2}
        \bm{v}^\top \operatorname*{smat}(\bm{A}_{\textnormal{sos}}^{(j)\top} {\bm{Z}^{(j)}}^{-1}\bm{\lambda})\bm{v}.
        \nonumber
    \end{equation}
    
    A lower bound for the perturbation term is obtained using $\sigma_{\textnormal{min}}(\bm{M}) \geq -\|\bm{M}\|_{\textnormal{F}}$ for any symmetric matrix $\bm{M}$,
    \begin{equation}
    \begin{aligned}
        \frac{\delta^{(j)}}{\|\bm{\lambda}\|_2}\,
        \bm{v}^{\top}
        \operatorname*{smat}\!\left(
            {\bm{A}_{\textnormal{sos}}^{(j)}}^{\top}
            {\bm{Z}^{(j)}}^{-1}
            \bm{\lambda}
        \right)
        \bm{v}
        &\geq
        -\frac{\delta^{(j)}}{\|\bm{\lambda}\|_2}
         \left\|
             {\bm{A}_{\textnormal{sos}}^{(j)}}^{\top}
             {\bm{Z}^{(j)}}^{-1}
             \bm{\lambda}
         \right\|_2
         \|\bm{v}\|_2^{2}
        \\
        &\geq
        -\frac{\delta^{(j)}}{
            \sqrt{
                \sigma_{\textnormal{min}}\!\left(
                    \bm{Z}^{(j)}
                \right)
            }
         }
         \|\bm{v}\|_2^{2},
    \end{aligned}
    \end{equation}

    where the last inequality follows from the norm bound derived above. 
    
    Therefore, using $\bm{v}^\top \operatorname*{smat}(\bar{\bm{\xi}}_{\textnormal{sos}}^{(j)})\bm{v} \geq \sigma_{\textnormal{min}}(\operatorname*{smat}(\bar{\bm{\xi}}_{\textnormal{sos}}^{(j)}))\|\bm{v}\|_2^2$ and substituting $\delta^{(j)}$, it holds that
    \begin{equation}
    \begin{aligned}
        \bm{v}^{\top}
        \operatorname*{smat}\!\left(
            \hat{\bm{\xi}}_{\textnormal{sos}}^{(j)}
        \right)
        \bm{v}
        &\geq
        \left(
            \sigma_{\min}\!\left(
                \operatorname*{smat}\!\left(
                    \bar{\bm{\xi}}_{\textnormal{sos}}^{(j)}
                \right)
            \right)
            -
            \frac{\delta^{(j)}}{
                \sqrt{
                    \sigma_{\textnormal{min}}\!\left(
                        \bm{Z}^{(j)}
                    \right)
                }
            }
        \right)
        \|\bm{v}\|_{2}^{2}
        \\
        &= 
        \frac{1}{2}\,
        \sigma_{\min}\!\left(
            \operatorname*{smat}\!\left(
                \bar{\bm{\xi}}_{\textnormal{sos}}^{(j)}
            \right)
        \right)
        \|\bm{v}\|_{2}^{2}
        \\
        &> 0,
    \end{aligned}
    \end{equation}

    which proves that $\hat{\bm{\xi}}_{\textnormal{sos}}^{(j)} \in \mathbb{S}^{h_j}_{\succ 0}$, and hence $\hat{\bm{\xi}} \in K$.

    Substituting $\hat{\bm{\xi}}$ into the affine constraint gives
    \begin{equation}
        \bm{A}\hat{\bm{\xi}} - \bm{b} 
        = \underbrace{\bm{A}_{\textnormal{f}}\bar{\bm{\xi}}_{\textnormal{f}} 
        + \sum_{j=1}^{m_{\textnormal{sos}}} \bm{A}_{\textnormal{sos}}^{(j)}\bar{\bm{\xi}}_{\textnormal{sos}}^{(j)} 
        - \bm{b}}_{=\,\bm{0}\;\textnormal{(Slater condition)}} 
        + \sum_{j=1}^{m_{\textnormal{sos}}} \frac{\delta^{(j)}}{\|\bm{\lambda}\|_2} 
        \underbrace{\bm{A}_{\textnormal{sos}}^{(j)}{\bm{A}_{\textnormal{sos}}^{(j)}}^{\top}}_{=\,\bm{Z}^{(j)}} 
        {\bm{Z}^{(j)}}^{-1}\bm{\lambda} 
        = \frac{\sum_{j=1}^{m_{\textnormal{sos}}}\delta^{(j)}}{\|\bm{\lambda}\|_2}\bm{\lambda},
        \nonumber
    \end{equation}
    so that the constraint term in the partial Lagrangian evaluates to
    \begin{equation}
        -\bm{\lambda}^\top(\bm{A}\hat{\bm{\xi}} - \bm{b}) = -\sum_{j=1}^{m_{\textnormal{sos}}}\delta^{(j)}\|\bm{\lambda}\|_2 = -\beta_1\|\bm{\lambda}\|_2.
        \label{eq:constraint-term}
    \end{equation}
    
    Since $\hat{\bm{\xi}}_{\textnormal{f}} = \bar{\bm{\xi}}_{\textnormal{f}}$ and $\|\hat{\bm{\xi}}_{\textnormal{sos}}^{(j)}\|_2$ is bounded above by a constant, all the remaining terms in $\mathcal{L}(\hat{\bm{\xi}}, \bm{\lambda})$ are upper bounded by a constant independent of $\bm{\lambda}$. Completing the square in the free variable terms using $\bm{Q} \in \mathbb{S}^{n_{\textnormal{f}}}_{\succ 0}$ gives
    \begin{equation}
        \frac{1}{2}\bar{\bm{\xi}}_{\textnormal{f}}^\top \bm{Q}\bar{\bm{\xi}}_{\textnormal{f}} + \bm{w}_{\textnormal{f}}^\top\bar{\bm{\xi}}_{\textnormal{f}} 
        = \frac{1}{2}\|\bm{Q}^{1/2}\bar{\bm{\xi}}_{\textnormal{f}} + \bm{Q}^{-1/2}\bm{w}_{\textnormal{f}}\|_2^2 
        - \frac{1}{2}\|\bm{Q}^{-1/2}\bm{w}_{\textnormal{f}}\|_2^2 
        \leq \frac{1}{2}\|\bm{Q}^{1/2}\bar{\bm{\xi}}_{\textnormal{f}} + \bm{Q}^{-1/2}\bm{w}_{\textnormal{f}}\|_2^2,
        \nonumber
    \end{equation}
    and applying the norm bound to the regularization terms, all remaining terms are bounded above by the constant $\beta_2$ in \cref{eq:beta2}.
    
    Combining \cref{eq:constraint-term} with the bound on the remaining terms, it holds that
    \begin{equation}
        g(\bm{\lambda}) \leq \mathcal{L}(\hat{\bm{\xi}}, \bm{\lambda}) \leq \beta_2 - \beta_1\|\bm{\lambda}\|_2,
        \nonumber
    \end{equation}
    with $\beta_1 > 0$ and $\beta_2 > 0$. Hence, $g(\bm{\lambda}) \to -\infty$ as $\|\bm{\lambda}\|_2 \to \infty$,
    which is precisely the definition of coercivity for a concave function.
\end{proof}

Now, we prove that the dual iterates generated by \cref{alg:first-order-solver} remain within a bounded set. This property is crucial for ensuring that error bounds are well-defined. The following lemma provides this guarantee by leveraging the coercivity of the dual function established in \cref{thm:coercivity}.

\begin{lemma}[Boundedness of Dual Iterates]
    Let $\{\bm{\lambda}^{(k)}\}_{k=0}^{\infty}$ be the sequence generated by the accelerated gradient in \cref{alg:first-order-solver} applied to the concave function $g$ with $L$-Lipschitz gradient. Since $g$ is coercive by \cref{thm:coercivity}, it follows that the sequence $\{\bm{\lambda}^{(k)}\}$ is uniformly bounded: there exists a constant $B > 0$ such that
    \begin{equation}
        \Vert\bm{\lambda}^{(k)}\Vert_{2} \leq B, \quad \text{for all } k \geq 0.
        \label{eq:bounded-iterates}
    \end{equation}
    \label{lem:bounded-iterates}
\end{lemma}
\begin{proof}
    According to \cref{thm:gradient-rate}, it follows that as $k \to \infty$
    \[
        \lim_{k \to \infty} g(\bm{\lambda}^{(k)}) = g(\bm{\lambda}^{\star}),
    \]
    where $\bm{\lambda}^{\star}$ is a global maximizer of $g$. Hence, there exists an integer $\hat{k}$ such that for all $k \geq \hat{k}$, it holds
    \[
        \vert g(\bm{\lambda}^{(k)}) - g(\bm{\lambda}^{\star}) \vert \leq 1.
    \]

    Assume, for contradiction, that $\{\bm{\lambda}^{(k)}\}$ is unbounded. Then there exists a subsequence $\{\bm{\lambda}^{(k_j)}\}_{j=1}^{\infty}$ such that
    \[
        \Vert\bm{\lambda}^{(k_j)}\Vert_{2} \to \infty \quad \text{as} \quad j \to \infty.
    \]
    
    Since $g$ is coercive, it follows that $g(\bm{\lambda}^{(k_j)}) \to -\infty$ and, for sufficiently large $j$, this yields
    \[
        g(\bm{\lambda}^{(k_j)}) <  g(\bm{\lambda}^{\star}) - 1.
    \]

    For all sufficiently large $j$, the condition $k_j \geq \hat{k}$ holds (since $k_j \to \infty$). Thus, $g(\bm{\lambda}^{(k_j)}) \geq  g(\bm{\lambda}^{\star}) - 1$, which contradicts the previous inequality. Therefore, the assumption of unboundedness is false and the sequence $\{\bm{\lambda}^{(k)}\}$ must be bounded.
\end{proof}

Now, we are ready to state the main convergence result of this paper. In this context, we define suboptimality by how close the objective value of \cref{opt:Reg-QCP} is from the optimum value of \cref{opt:QCP}. Mathematically, we analyze the difference between objective functions
\[
    f_{\rho}(\mathcal{S}_{\bm{\xi}}(\bm{\lambda}^{(N)})) - f(\bm{\xi}^{\star}),
\]
where $\bm{\xi}^{\star}$ is the optimal solution of the original problem of \cref{opt:QCP} and $\mathcal{S}_{\bm{\xi}}(\bm{\lambda}^{(N)})$ is the primal solution returned by \cref{alg:first-order-solver} after $N$ iterations for the regularized problem of \cref{opt:Reg-QCP}.

\begin{theorem}[non-asymptotic Bound for Suboptimality]
    Consider the original quadratic conic program of \cref{opt:QCP} with objective
    \[
        f(\bm{\xi}) = \frac{1}{2}\bm{\xi}_{\textnormal{f}}^{\top} \bm{Q} \bm{\xi}_{\textnormal{f}} + \bm{w}^{\top} \bm{\xi},
    \]
    and its regularized version of \cref{opt:Reg-QCP} with objective
    \[
        f_{\rho}(\bm{\xi}) = \frac{1}{2}\bm{\xi}_{\textnormal{f}}^{\top}\bm{Q}\bm{\xi}_{\textnormal{f}} + \bm{w}^{\top} \bm{\xi} + \frac{\rho}{2} \sum_{j=1}^{m_{\textnormal{sos}}} \Vert \bm{\xi}_{\textnormal{sos}}^{(j)} \Vert_{\textnormal{2}}^{2}.    
    \]
    Let $\bm{Q} \in \mathbb{S}_{\succ 0}^{n_{\textnormal{f}}}$ and the feasible set of \cref{opt:Reg-QCP} be nonempty with Slater point $\bar{\bm{\xi}} \in \operatorname*{int}{K}$ satisfying $\bm{A} \bar{\bm{\xi}} = \bm{b}$.
    Let $\mathcal{S}_{\bm{\xi}}(\bm{\lambda}^{(N)})$ and $\bm{\lambda}^{(N)}$ denote the primal and dual iterates, respectively, returned by \cref{alg:first-order-solver} after $N$ gradient steps with $\eta = 1/L$, where $L$ is the Lipschitz constant in \cref{eq:lipschitz-constant}. Let $\bm{\xi}^{\star}$ be an optimal solution of \cref{opt:QCP} and $\mathcal{S}_{\bm{\xi}}(\bm{\lambda}^{\star})$ an optimal solution of \cref{opt:Reg-QCP}.
    
    Then, for all $N \geq 0$,
    \begin{equation}
        f(\mathcal{S}_{\bm{\xi}}(\bm{\lambda}^{(N)})) - f(\bm{\xi}^{\star}) \leq \frac{2 B L \|\bm{\lambda}^{(0)} - \bm{\lambda}^\star\|_2}{(N + 1)} + \rho \sum_{j=1}^{m_{\textnormal{sos}}} \Vert {\bm{\xi}_{\textnormal{sos}}^{(j)}}^{\star} \Vert_{\textnormal{2}}^{2}.
    \end{equation}
    \label{thm:suboptimality-fixed-rho}
\end{theorem}
\begin{proof}
    The proof proceeds in three steps: relate the regularized objective at the recovered primal $\mathcal{S}_{\bm{\xi}}(\bm{\lambda}^{(N)})$ to the dual function value $g(\bm{\lambda}^{(N)})$, bound the regularized suboptimality $f_{\rho}(\mathcal{S}_{\bm{\xi}}(\bm{\lambda}^{(N)})) - f_{\rho}(\mathcal{S}_{\bm{\xi}}(\bm{\lambda}^{\star}))$ using the convergence of the dual problem, and transfer the bound to the original objective $f$ via the regularization error.

    \textbf{Step 1.} From the definitions of $\mathcal{S}_{\bm{\xi}_{\textnormal{f}}}(\bm{\lambda})$ and $\mathcal{S}_{\bm{\xi}_{\textnormal{sos}}^{(j)}}(\bm{\lambda})$ in \cref{eq:free-sol} and \cref{eq:sos-sol}, respectively, the returned primal $\mathcal{S}_{\bm{\xi}}(\bm{\lambda}^{(N)})$ minimizes the partial Lagrangian $\mathcal{L}(\bm{\xi}, \bm{\lambda}^{(N)})$ over $\bm{\xi} \in K$. Hence, evaluating the Lagrangian at $\mathcal{S}_{\bm{\xi}}(\bm{\lambda}^{(N)})$ gives
    \begin{equation}
        \mathcal{L}(\mathcal{S}_{\bm{\xi}}(\bm{\lambda}^{(N)}), \bm{\lambda}^{(N)}) = f_{\rho}(\mathcal{S}_{\bm{\xi}}(\bm{\lambda}^{(N)})) - {\bm{\lambda}^{(N)}}^{\top}( \bm{A} \mathcal{S}_{\bm{\xi}}(\bm{\lambda}^{(N)}) - \bm{b} ) = g(\bm{\lambda}^{(N)}).
        \label{eq:convergence-aux-1}
    \end{equation}

    The gradient of the dual function, derived in \cref{eq:gradient-dual} and expressed in \cref{eq:gradient-dual-feasibility}, satisfies
    \begin{equation}
        \nabla g(\bm{\lambda}^{(N)}) = - ( \bm{A} \mathcal{S}_{\bm{\xi}}(\bm{\lambda}^{(N)}) - \bm{b} ).
        \label{eq:convergence-aux-2}
    \end{equation}

    Combining \cref{eq:convergence-aux-1} and \cref{eq:convergence-aux-2} yields
    \begin{equation}
        f_{\rho}(\mathcal{S}_{\bm{\xi}}(\bm{\lambda}^{(N)})) = g(\bm{\lambda}^{(N)}) - {\bm{\lambda}^{(N)}}^{\top} \nabla g(\bm{\lambda}^{(N)}),
        \label{eq:convergence-aux-3}
    \end{equation}
    which concludes the first step of the proof.

    \textbf{Step 2.} Now, because \cref{opt:Reg-QCP} is a convex optimization problem satisfying Slater's condition, strong duality holds. Let $\bm{\lambda}^{\star}$ be a dual optimal solution, which exists and is finite due to coercivity established in \cref{thm:coercivity}. Then,
    \begin{equation}
        f_{\rho}(\mathcal{S}_{\bm{\xi}}(\bm{\lambda}^{\star})) = g(\bm{\lambda}^{\star})
        \label{eq:convergence-aux-4}
    \end{equation}

    Using \cref{eq:convergence-aux-3} and \cref{eq:convergence-aux-4}, the suboptimality of the regularized problem is decomposed as
    \begin{equation}
        f_{\rho}(\mathcal{S}_{\bm{\xi}}(\bm{\lambda}^{(N)})) - f_{\rho}(\mathcal{S}_{\bm{\xi}}(\bm{\lambda}^{\star})) = ( g(\bm{\lambda}^{(N)}) - g(\bm{\lambda}^{\star}) ) - {\bm{\lambda}^{(N)}}^{\top} \nabla g(\bm{\lambda}^{(N)})
        \label{eq:convergence-aux-5}
    \end{equation}

    \cref{thm:gradient-rate} provides the accelerated dual convergence bound
    \begin{equation}
        g(\bm{\lambda}^\star) - g(\bm{\lambda}^{(N)}) \leq \frac{2L \|\bm{\lambda}^{(0)} - \bm{\lambda}^\star\|_2^2}{(N+1)^2}.
        \nonumber
    \end{equation}

    Since $g$ is concave (see \cref{lem: dual concave and differentiable}) and $\bm{\lambda}^\star$ is a maximizer, it follows that
    \begin{equation}
        g(\bm{\lambda}^{(N)}) - g(\bm{\lambda}^\star) \leq 0 \leq g(\bm{\lambda}^\star) - g(\bm{\lambda}^{(N)}) \leq \frac{2L \|\bm{\lambda}^{(0)} - \bm{\lambda}^\star\|_2^2}{(N+1)^2}.
        \label{eq:convergence-aux-6}
    \end{equation}
    Therefore, by substituting \cref{eq:convergence-aux-6} into \cref{eq:convergence-aux-5}, one has
    \begin{equation}
        f_{\rho}(\mathcal{S}_{\bm{\xi}}(\bm{\lambda}^{(N)})) - f_{\rho}(\mathcal{S}_{\bm{\xi}}(\bm{\lambda}^{\star})) \leq  - {\bm{\lambda}^{(N)}}^{\top} \nabla g(\bm{\lambda}^{(N)}),
        \nonumber
    \end{equation}
    and, by using the Cauchy–Schwarz inequality to bound the last term of the right-hand side, yields
    \begin{equation}
        f_{\rho}(\mathcal{S}_{\bm{\xi}}(\bm{\lambda}^{(N)})) - f_{\rho}(\mathcal{S}_{\bm{\xi}}(\bm{\lambda}^{\star})) \leq  \| \bm{\lambda}^{(N)} \|_2 \| \nabla g(\bm{\lambda}^{(N)}) \|_2,
        \label{eq:convergence-aux-7}
    \end{equation}
    which concludes the second step of the proof.

    \textbf{Step 3.} For any $\bm{\xi} \in K$,
    \begin{equation}
        f_{\rho}(\bm{\xi}) = f(\bm{\xi}) + \frac{\rho}{2} \sum_{j=1}^{m_{\textnormal{sos}}} \Vert \bm{\xi}_{\textnormal{sos}}^{(j)} \Vert_{\textnormal{2}}^{2}.
        \nonumber
    \end{equation}
    Applying this to $\mathcal{S}_{\bm{\xi}}(\bm{\lambda}^{(N)})$ and to an optimal solution $\bm{\xi}^{\star}$ of the original problem of \cref{opt:QCP} yields
    \begin{equation}
        f(\mathcal{S}_{\bm{\xi}}(\bm{\lambda}^{(N)})) - f(\bm{\xi}^{\star}) = f_{\rho}(\mathcal{S}_{\bm{\xi}}(\bm{\lambda}^{(N)})) - f_{\rho}(\bm{\xi}^{\star}) - \frac{\rho}{2} \sum_{j=1}^{m_{\textnormal{sos}}} \Vert {\bm{\xi}_{\textnormal{sos}}^{(j)}}^{(N)} \Vert_{\textnormal{2}}^{2} + \frac{\rho}{2} \sum_{j=1}^{m_{\textnormal{sos}}} \Vert {\bm{\xi}_{\textnormal{sos}}^{(j)}}^{\star} \Vert_{\textnormal{2}}^{2}.
        \nonumber
    \end{equation}
    Because the penalty term is nonnegative, one has
    \begin{equation}
        f(\mathcal{S}_{\bm{\xi}}(\bm{\lambda}^{(N)})) - f(\bm{\xi}^{\star}) \leq f_{\rho}(\mathcal{S}_{\bm{\xi}}(\bm{\lambda}^{(N)})) - f_{\rho}(\bm{\xi}^{\star}) + \frac{\rho}{2} \sum_{j=1}^{m_{\textnormal{sos}}} \Vert {\bm{\xi}_{\textnormal{sos}}^{(j)}}^{\star} \Vert_{\textnormal{2}}^{2}.
        \label{eq:convergence-aux-8}
    \end{equation}
    Now, note that $\bm{\xi}^{\star}$ is feasible for \cref{opt:Reg-QCP}, so the optimality of $\mathcal{S}_{\bm{\xi}}(\bm{\lambda}^{\star})$ for the regularized problem implies
    \begin{equation}
        f_{\rho}(\mathcal{S}_{\bm{\xi}}(\bm{\lambda}^{\star})) \leq f_{\rho}(\bm{\xi}^{\star}) = f(\bm{\xi}^{\star}) + \frac{\rho}{2} \sum_{j=1}^{m_{\textnormal{sos}}} \Vert {\bm{\xi}_{\textnormal{sos}}^{(j)}}^{\star} \Vert_{\textnormal{2}}^{2}.
        \nonumber
    \end{equation}
    Thus
    \begin{equation}
        f_{\rho}(\mathcal{S}_{\bm{\xi}}(\bm{\lambda}^{(N)})) - f_{\rho}(\bm{\xi}^{\star}) \leq f_{\rho}(\mathcal{S}_{\bm{\xi}}(\bm{\lambda}^{(N)})) - f_{\rho}(\mathcal{S}_{\bm{\xi}}(\bm{\lambda}^{\star})) + \frac{\rho}{2} \sum_{j=1}^{m_{\textnormal{sos}}} \Vert {\bm{\xi}_{\textnormal{sos}}^{(j)}}^{\star} \Vert_{\textnormal{2}}^{2},
        \label{eq:convergence-aux-9}
    \end{equation}
    and by combining \cref{eq:convergence-aux-8} and \cref{eq:convergence-aux-9}, one has
    \begin{equation}
        f(\mathcal{S}_{\bm{\xi}}(\bm{\lambda}^{(N)})) - f(\bm{\xi}^{\star}) \leq f_{\rho}(\mathcal{S}_{\bm{\xi}}(\bm{\lambda}^{(N)})) - f_{\rho}(\mathcal{S}_{\bm{\xi}}(\bm{\lambda}^{\star})) + \rho \sum_{j=1}^{m_{\textnormal{sos}}} \Vert {\bm{\xi}_{\textnormal{sos}}^{(j)}}^{\star} \Vert_{\textnormal{2}}^{2}.
        \nonumber
    \end{equation}
    Applying the upper bound in \cref{eq:convergence-aux-7}, it holds
    \begin{equation}
        f(\mathcal{S}_{\bm{\xi}}(\bm{\lambda}^{(N)})) - f(\bm{\xi}^{\star}) \leq \| \bm{\lambda}^{(N)} \|_2 \| \nabla g(\bm{\lambda}^{(N)}) \|_2 + \rho \sum_{j=1}^{m_{\textnormal{sos}}} \Vert {\bm{\xi}_{\textnormal{sos}}^{(j)}}^{\star} \Vert_{\textnormal{2}}^{2}.
        \nonumber
    \end{equation}
    Using \cref{lem:bounded-iterates} to bound the dual iterate and the bound in \cref{eq:convergence-feas} from the proof of \cref{thm:feasibility} gives
    \begin{equation}
        f(\mathcal{S}_{\bm{\xi}}(\bm{\lambda}^{(N)})) - f(\bm{\xi}^{\star}) \leq  \frac{2 B L \|\bm{\lambda}^{(0)} - \bm{\lambda}^\star\|_2}{(N + 1)} + \rho \sum_{j=1}^{m_{\textnormal{sos}}} \Vert {\bm{\xi}_{\textnormal{sos}}^{(j)}}^{\star} \Vert_{\textnormal{2}}^{2},
    \end{equation}
    concluding the last step of the proof.
\end{proof}

The suboptimality bound in \cref{thm:suboptimality-fixed-rho} demonstrates that the distance between the regularized primal solution of \cref{opt:Reg-QCP} and the solution of the original \cref{opt:QCP} decreases as the number of iterations $N$ grows, indicating accelerated convergence of the dual method. However, the bound also contains terms proportional to $\rho$, implying that the regularized problem must approach the unregularized one for the approximation to be exact. Intuitively, as $\rho \to 0$, the objective of \cref{opt:Reg-QCP} is equal with that of \cref{opt:QCP}, but the Lipschitz constant of the dual gradient grows as $\mathcal{O}(1/\rho)$, making the dual problem increasingly slow to solve: to achieve very small error, $N$ must increase much faster relative to $1/\rho$, which guarantees that the accelerated method compensates for the larger Lipschitz constant or, equivalently, for the smaller stepsize. Such a relationship indicates that vanishing regularization increases accuracy compared to the original problem while requiring additional dual iterations to maintain convergence guarantees.

\textbf{Termination Criteria.} \cref{alg:first-order-solver} terminates when both the primal feasibility residual and the duality gap fall below absolute and relative tolerances. At iteration $k$, primal feasibility is measured and termination requires
\begin{equation}
    r_{\textnormal{prim}}^{(k)} < \epsilon_{\textnormal{abs}} + \epsilon_{\textnormal{rel}} \max\{ \Vert \bm{y}^{(k)} \Vert_{\infty}, 1 \},
    \nonumber
\end{equation}
where $r_{\textnormal{prim}}^{(k)} = \Vert \nabla g (\bm{y}^{(k)}) \Vert_{\infty}$, $\epsilon_{\textnormal{abs}} > 0$, and $\epsilon_{\textnormal{rel}} > 0$ are the primal residual and user supplied tolerances, respectively.

Although the primal feasibility residual already guarantees that the recovered point is nearly feasible, one may also wish to ensure near-optimality. The duality gap for the regularized problem of \cref{opt:Reg-QCP} is given by \cref{eq:convergence-aux-1},
\begin{equation}
    \vert f_{\rho}(\mathcal{S}_{\bm{\xi}}(\bm{\lambda})) - g(\bm{\lambda}) \vert = \vert \bm{\lambda}^{\top}( \bm{A} \mathcal{S}_{\bm{\xi}}(\bm{\lambda}) - \bm{b} ) \vert = \vert \bm{\lambda}^{\top}  \nabla g(\bm{\lambda}) \vert,
    \nonumber
\end{equation}
and the algorithm stops when
\begin{equation}
     r_{\textnormal{gap}}^{(k)} < \epsilon_{\textnormal{abs}} + \epsilon_{\textnormal{rel}} \max \{ \vert f_{\rho}(\bm{y}^{(k)}) \vert, \vert g(\bm{y}^{(k)}) \vert \},
    \nonumber
\end{equation}
which reflects a relative duality gap, where $r_{\textnormal{gap}}^{(k)} = \vert {\bm{y}^{(k)}}^{\top} \nabla g(\bm{y}^{(k)}) \vert$. In general, it is noted that checking only the primal feasibility residual is often sufficient, since the convergence of the dual objective (\cref{thm:gradient-rate}) and the boundedness of the dual variables (\cref{lem:bounded-iterates}) imply that the gap also becomes small.

These criteria mirror those used in modern conic solvers such as COSMO, SCS, and Clarabel, where termination is based on feasibility of affine constraints and a certificate of near-optimality. Because primal variables are recovered in closed form from the dual iterate, feasibility and optimality can be monitored directly through the gradient of the dual function. Adaptive restart ensures stability of the accelerated scheme, and the algorithm stops either when tolerances are met or when the maximum number of iterations is reached.

\textbf{Adaptive Restart.} The accelerated scheme uses an adaptive restart mechanism to maintain stability and prevent oscillations typical of Nesterov’s method. At iteration 
$k$, restart is triggered when the gradient direction becomes misaligned with the extrapolated step, i.e., when ${\nabla g(\bm{y}^{(k-1)})}^{\top} ( \bm{\lambda}^{(k)} - \bm{\lambda}^{(k-1)} ) < 0$. In this case, the momentum parameter is reset to $t^{(k)} = 1$, and the extrapolated point is replaced by the current dual iterate, $\bm{y}^{(k)} = \bm{\lambda}^{(k)}$, warm-starting the accelerated gradient. This heuristic \cite{o2015adaptive} helps to mitigate oscillatory behavior and can lead to faster practical convergence, especially when the objective is not strongly convex or when the Lipschitz constant is over-estimated. In practice, it often reduces the number of iterations required to reach a given tolerance, as the algorithm adaptively resets when progress stalls.

\section{Numerical Experiments}
\label{sec:experiments}

In this section, benchmark examples are presented to show how can one employ \cref{alg:first-order-solver} to solve the problems in the form of \cref{opt:QSOS}. At the current state, \cref{alg:first-order-solver} is implemented in MATLAB \cite{matlab} due to being fast to develop and debug. All programs are parsed into an optimization problem in the form of \cref{opt:QCP}. Specifically, the quadratic objective is parsed manually, while the constraints are parsed using SOSTOOLS \cite{sostools4}. Then, to compare different solving approaches, the original problem of \cref{opt:QCP} is lifted into Opt. \cref{opt:standard-QCP} (which is solved by SCS \cite{o2021operator}), or lifted into \cref{opt:socp} and \cref{opt:schur} which are solved by MOSEK \cite{mosek}. Only \cref{alg:first-order-solver} solves the problem of \cref{opt:QCP} without lifting the decision variables or the constraints space.

The comparison between \cref{alg:first-order-solver} and IPMs focuses primarily on scalability: IPMs are known to suffer from high memory and computational costs as problem dimensions grow, whereas the proposed method maintains a much smaller memory usage. In contrast, the comparison with SCS (another first-order conic solver) centers on computational efficiency: since scalability is less of a differentiator, they are compared in terms of solution time and achievable accuracy. In SCS and \cref{alg:first-order-solver}, $\epsilon_{\textnormal{abs}} = 10^{-12}$ and $\epsilon_{\textnormal{rel}} = 10^{-12}$ are set for full accuracy of solutions and a maximum of $10^{5}$ iterations is defined. In MOSEK, default termination criteria parameters are used. The maximum solve time for all solvers is set to 600 seconds (10 minutes) and if a solver fail (due to numerical problems or memory exhaustion) the maximum allowable solve time is set. Next, performance metrics are defined to compare the proposed approach against other solvers.

\textbf{Performance Metrics.} In this section, the standard benchmark convention is followed using performance profiles to compare solve time and the ratio of problems solved \cite{dolan2002benchmarking}. The first metric uses the relative performance ratio, which is defined as the time $t$ to solve a problem $p$ using a solver $s$ divided by the fastest solve time across all solvers, $u_{p,s} = t_{p,s} / \min_{s} t_{p,s}$. For the ratio of $N$ problems solved relative to the fastest solve time, define $f: \mathbb{R}_{\geq 0} \to [0,1]$ such that $f_{s}(\tau) = 1/N \sum_{p}\mathcal{I}_{\leq \tau}(u_{p,s})$, where
\begin{equation}
    \mathcal{I}_{\leq \tau}(u_{p,s}) =
    \begin{cases}
      1 & \text{if $t_{p,s} \leq \tau$},\\
      0 & \text{otherwise}.
    \end{cases} 
    \nonumber
\end{equation}

This metric measures the ratio of problems solved by solver $s$ within the time $\tau$. By varying $\tau$ as a multiple of the best solve time, a relative performance profile is constructed using $f_{s}(\tau)$.

The second metric uses the shifted geometric mean of solve times defined by
\begin{equation}
    \text{SGM}_{s} = {\left( \prod_{i=1}^{N} (t_{i,s} + 1) \right)}^{1/N} - 1,
    \nonumber
\end{equation}
and consists of the normalized shifted geometric mean, $\text{NSGM}_{s} = \text{SGM}_{s} / \min_{s} \text{SGM}_{s}$.

This metric measures the typical relative performance of a solver against the best on each problem. It provides a single-number summary of how close a solver is to the best ($\text{NSGM}_{s} = 1$), but it is worth noting that performance profiles remain the standard point of reference to evaluate performance.

\textbf{Constrained Regression Problems.} Consider the following regression problem
\begin{equation}
    \begin{aligned}
        \min_{p}~& \sum_{i=1}^{D}\left(p(\bm{x}^{(i)}) - y^{(i)}\right)^{2} \quad \text{s.t.}~& p(\bm{x})~\textnormal{is sum of squares},\\
    \end{aligned}
    \label{opt:ex1}
\end{equation}
where the solution is a sum-of-squares polynomial $p^{\star}(\bm{x})$ with coefficients ${\bm{\xi}_{\textnormal{f}}}^{\star} \in \mathbb{R}^{n_{\textnormal{f}}}$ that fits the $D \in \mathbb{N}$ data points. For any polynomial $p(\bm{x}) = \bm{z}_{2d}(\bm{x})^{\top} \bm{\xi}_{\textnormal{f}}$ (of the same degree of $p^{\star}(\bm{x})$), the objective becomes $\sum_{i=1}^{D}\left(p(\bm{x}^{(i)}) - y^{(i)}\right)^{2} = \frac{1}{2}\bm{\xi}_{\textnormal{f}}^{\top}\bm{Q}\bm{\xi}_{\textnormal{f}} + \bm{w}_{\textnormal{f}}^{\top}\bm{\xi}_{\textnormal{f}} + \text{constant}$, where $\bm{Q} = 2\sum_{i=1}^{D}\bm{z}_{2d}(\bm{x}^{(i)})\bm{z}_{2d}(\bm{x}^{(i)})^{\top}$ and $\bm{w}_{\textnormal{f}} = -2\sum_{i=1}^{D}y^{(i)}\bm{z}_{2d}(\bm{x}^{(i)})$. The constant term is omitted as it does not affect minimization. Hence, this problem is computationally equivalent to the form
\begin{equation}
    \begin{split}
        \min_{\bm{\xi}_{\textnormal{f}}, \bm{\Xi}^{(1)}} \quad &\frac{1}{2}\bm{\xi}_{\textnormal{f}}^{\top} \bm{Q} \bm{\xi}_{\textnormal{f}} + \bm{w}_{\textnormal{f}}^{\top} \bm{\xi}_{\textnormal{f}}\\
        \text{s.t.} \quad &\bm{z}_{2d}(\bm{x})^{\top} \bm{\xi}_{\textnormal{f}} = \bm{z}_{d}(\bm{x})^{\top} \bm{\Xi}^{(1)} \bm{z}_{d}(\bm{x}),\\
        & \bm{\Xi}^{(1)} \in \mathbb{S}_{+}^{(n_x + d)!/(n_x!~d!)}.\\
    \end{split}
    \nonumber
\end{equation}

By following the procedure presented in Section \ref{sec:lifting}, one can reformulate this problem into the form of \cref{opt:QCP}. To construct sets of benchmark problems, randomly generated sum-of-squares polynomials of different degrees $2d$ in different dimension sizes $n_{x}$ define the problem size. The set of data points $\{\bm{x}^{(i)}, y^{(i)}\}_{i = 1}^{D}$ is defined by randomly sampling $\bm{x}^{(i)} \in [-1,1]^{n_{x}}$ and generating $y^{(i)} = p^{\star}(\bm{x}^{(i)})$ by the ground-truth solution, whose coefficients are ${\bm{\xi}_{\textnormal{f}}}^{\star} \in [-1,1]^{n_{\textnormal{f}}}$. The scaling of $\bm{x}^{(i)}$, ${\bm{\xi}_{\textnormal{f}}}^{\star}$, and the number $D = 2 n_{\textnormal{f}}$ of data points are defined to ensure $\bm{Q} \in \mathbb{S}_{\succ 0}^{n_{\textnormal{f}}}$ and to avoid ill conditioning as much as possible.

In this numerical comparison, benchmark problems where $1000 \leq n_{\textnormal{f}} \leq 20000$ are considered, with this range defining the problems' size. For the set of problems, $10$ different solution polynomials are defined for $24$ different combinations of $(n_{x}, 2d)$, resulting in $240$ problems in total.

The numerical results are summarized in \cref{tab:ex1} and \cref{fig:ex1}. \cref{tab:ex1} reports the normalized shifted geometric mean (NSGM) of solve times, the failure rate, and the averages of suboptimality and primal feasibility (computed just for reference and only over successfully solved instances). \cref{alg:first-order-solver}, with regularization parameter $\rho = 10^{-6}$, solves all $240$ problems, achieving a zero failure rate and the best NSGM of 1. SCS has an NSGM of $1.4013$, about $40\%$ slower than \cref{alg:first-order-solver} on average, but fails on $8.33\%$ of the problems. Both MOSEK reformulations (SOCP and SDP) exhibit high failure rates ($71.67\%$ and $75\%$, respectively) and infinite NSGM due to numerous timeouts and memory errors. In terms of solution quality, \cref{alg:first-order-solver} and SCS attain comparable average suboptimality (averaged $4.25 \times 10^{-8}$) and primal feasibility (averaged $1.33 \times 10^{-13}$) on the problems they both solve. MOSEK, when it succeeds, achieves lower accuracy (suboptimality as high as $0.24$ for \cref{opt:socp} and $\approx 5 \times 10^{-6}$ for \cref{opt:schur}) and has  longer runtimes, showing poor reliability.

\begin{table}[htb]
\centering
\resizebox{\textwidth}{!}{%
\begin{tabular}{l | c | c | c | c}
\toprule
Method & Alg.~\cref{alg:first-order-solver}, \cref{opt:QCP} &
SCS, Opt.~\cref{opt:standard-QCP} &
MOSEK, \cref{opt:socp} &
MOSEK, \cref{opt:schur} \\
\midrule
NSGM & $1$ & $1.4013$ & $\infty$ & $\infty$\\
\midrule
Failure rate (\%) & $0$ & $8.3333$ & $71.6667$ & $75$\\
\midrule
Avg.\ suboptimality &
$5.1499 \times 10^{-8}$ &
$3.3482 \times 10^{-8}$ &
$0.2473$ &
$4.9624 \times 10^{-6}$\\
\midrule
Avg.\ primal feasibility &
$2.5635 \times 10^{-13}$ &
$1.0619 \times 10^{-14}$ &
$2.0277 \times 10^{-7}$ &
$2.5030 \times 10^{-10}$\\
\bottomrule
\end{tabular}
}
\caption{Performance comparison on 240 constrained regression problems. Reported are the normalized shifted geometric mean (NSGM) of solve times (lower is better, $1$ indicates the fastest on average), failure rate (percentage of problems not solved within the $10$-minute time limit or due to memory exhaustion), and averages of suboptimality and primal feasibility (computed only over successfully solved instances).}
\label{tab:ex1}
\end{table}

\cref{fig:ex1} presents the performance profile. The curve for \cref{alg:first-order-solver} rises quickly, reaching about $80\%$ of problems solved at a relative performance ratio $\tau > 50$, meaning that $80\%$ of the problems are solved in at most $50$ times the fastest solve time, which corresponds to approximately $15$ seconds. For $100 < \tau < 250$, about $90\%$ of the problems are solved faster with \cref{alg:first-order-solver} than with any of the other solvers. SCS's curve flattens at $91.67\%$ because it fails on the remaining $8.33\%$, whereas \cref{alg:first-order-solver} can solve all problems within the $10$-minute time limit. Both MOSEK variants climb slowly and do not exceed $30\%$ even at large $\tau$, confirming their poor scalability. Overall, the results demonstrate that \cref{alg:first-order-solver} is a robust, memory-efficient, and practically accurate solver for large-scale quadratic SOS programs, outperforming MOSEK in reliability and SCS in the fraction of problems solved, presenting competitive speed and accuracy.

\begin{figure}[htb]
    \centering
    \includegraphics[width=\textwidth]{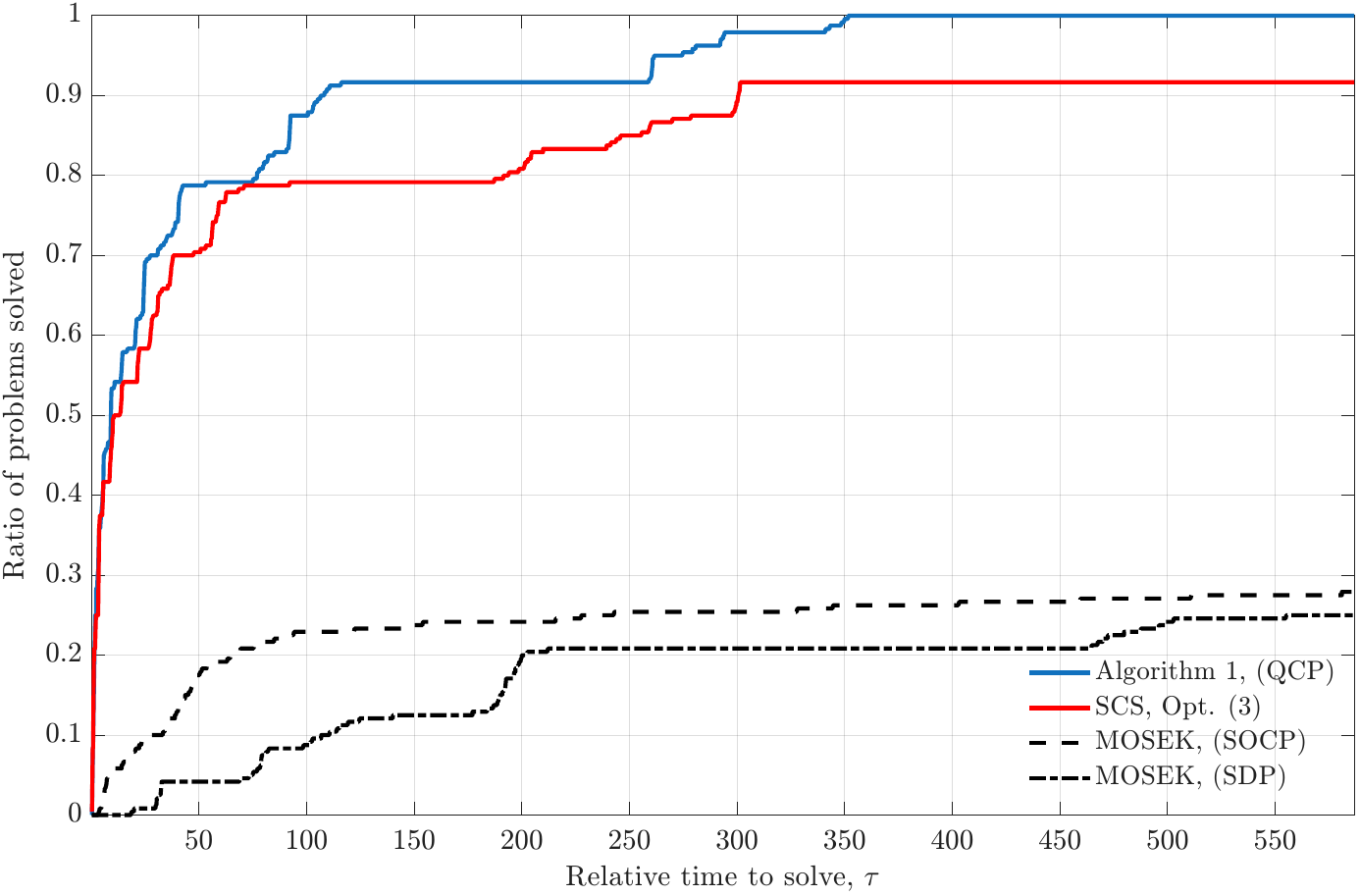}
    \caption{Relative performance profiles to capture the behavior of solvers in different scale problems. The fraction of problems solved is computed for each multiple $\tau$ of the best solve time.}
    \label{fig:ex1}
\end{figure}

\section{Conclusion}
\label{sec:conclusion}

This work introduced a lifting-free regularization method for quadratic Sum-Of-Squares (QSOS) programs. By adding a norm penalty to SOS-constrained variables, the method preserves the original conic structure and yields closed-form primal updates, reducing the problem to an unconstrained dual with Lipschitz-continuous gradient. The dual can be efficiently maximized using accelerated first-order methods, and the theoretical analysis establishes concavity, differentiability, smoothness, and non-asymptotic recovery of the original solution of \cref{opt:QCP}. \cref{thm:feasibility} guarantees that the primal feasibility residual vanishes at the rate $\mathcal{O}(1/N)$, and \cref{thm:suboptimality-fixed-rho} gives a complete bound on the suboptimality of the recovered primal solution in terms of the number of iterations, the regularization parameter $\rho$, and the distance to the optimal dual point. These guarantees are unconditional and hold for all problems without any additional assumptions.

Numerical experiments on constrained regression problems demonstrate that the method solves all instances, handles larger problems than interior-point reformulations, and can achieve faster performance than SCS and MOSEK with memory scaling only in the number of equality constraints. These results show that lifting-free regularization provides a practical and scalable alternative for large-scale quadratic SOS optimization.

\textbf{Future Work.} Although the current implementation of \cref{alg:first-order-solver} already demonstrates competitive performance, several improvements could further enhance its efficiency and robustness. The dual problem can be solved using alternative first-order methods that may converge in significantly fewer iterations, especially for ill-conditioned problems. A direct implementation of the algorithm in a compiled language would eliminate the overhead of the MATLAB interpreter and likely improve runtime. Future work could also systematically investigate the use of adaptive regularization (varying $\rho$ during the solve) and warm-starting strategies to handle even larger and more difficult problems. Finally, the robustness observed in the benchmarks suggests that the method is well suited for problems where other solvers fail. The present work thus provides a solid foundation for a family of simple, memory-efficient, and robust solvers.





\bibliographystyle{siamplain}
\bibliography{references}

\end{document}

%% file: ex_shared.tex
\usepackage{amsmath}      
\usepackage{amssymb}      
\usepackage{amsfonts}     
\usepackage{mathtools}    
\usepackage{graphicx}     
\usepackage{bm}           
\usepackage{booktabs}     
\usepackage{algorithm}    
\usepackage{algpseudocode}
\usepackage{amsopn}       

\ifpdf
  \DeclareGraphicsExtensions{.pdf,.png,.jpg}
\else
  \DeclareGraphicsExtensions{.eps}
\fi

\DeclareMathOperator*{\argmin}{arg\,min}

\newsiamremark{remark}{Remark}
\newsiamremark{example}{Example}
\newsiamremark{assumption}{Assumption}

\headers{Lifting-Free Quadratic Sum-Of-Squares Programming}
        {Gabriel F. Machado, Ross Drummond, and Morgan Jones}

\title{Lifting-Free Quadratic Sum-Of-Squares Programming%
\thanks{Submitted to the editors DATE.}}

\author{
Gabriel F. Machado
\and
Ross Drummond
\and
Morgan Jones
\thanks{School of Electrical and Electronic Engineering, University of Sheffield, Amy Johnson Building, Mappin Street, Sheffield, S1 3JD, UK
(\email{\{gfmachado1, ross.drummond, morgan.jones\}@sheffield.ac.uk}).}
}